\documentclass{amsart}
\usepackage{comment}
\usepackage{faktor}
\usepackage{array}

\usepackage{tikz-cd}
\usepackage{multicol}

\newcommand{\p}{\frac{1}{2}}

\newcommand{\ds}[2]{\delta([\nu^{#1} \rho,\nu^{#2} \rho])}

\newcommand{\st}[4]
{\ds{#1}{#2} \times \ds{#3}{#4} \rtimes \sigma }

\newcommand{\A}{\st{-a}{c}{\p}{b}   }
\newcommand{\B}{\st{\p}{b}{-a}{c}   }
\newcommand{\C}{\st{\p}{b}{-c}{a}   }
\newcommand{\D}{\st{-c}{a}{\p}{b}   }
\newcommand{\E}{\st{-c}{a}{-b}{-\p} } 

\newcommand{\F}{\st{\p}{b}{-a}{c}   }
\newcommand{\G}{\st{-a}{c}{\p}{b}   }
\newcommand{\I}{\st{-a}{c}{-b}{-\p}   }
\newcommand{\J}{\st{-b}{-\p}{-a}{c}   }
\newcommand{\K}{\st{-b}{-\p}{-c}{a}   }

\newtheorem{theorem}{Theorem}[section]
\newtheorem*{theorem*}{Theorem}
\newtheorem{lemma}[theorem]{Lemma}
\newtheorem{corollary}[theorem]{Corollary}

\numberwithin{equation}{section}
\numberwithin{figure}{section}
\numberwithin{table}{section}

\newtheorem{proposition}[theorem]{Proposition}

\theoremstyle{definition}

\theoremstyle{remark}

\usepackage{epsfig,amscd,amssymb,amsxtra,amsmath,amsthm}
\usepackage{url}

\begin{document}

\title[A one half cuspidal reducibility]
{
Parabolic induction from two segments, linked under contragredient,  with a one half cuspidal reducibility, a special case}

\author{Igor Ciganovi\'{c}}
\address{Igor Ciganovi\'{c}, Department of Mathematics, Faculty of Science, University of Zagreb, Bijeni\v{c}ka cesta 30, HR-10000 Zagreb, Croatia}
\email{igor.ciganovic@math.hr}


\begin{abstract}
 In this paper, we determine the composition series of the induced representation   
 $\delta([\nu^{-a}\rho,\nu^c\rho])\times \delta([\nu^\frac{1}{2}\rho,\nu^b\rho])\rtimes \sigma$ where 
  $a, b, c \in \mathbb{Z}+\frac{1}{2}$ such that $\frac{1}{2}\leq a < b  < c$,  
 $\rho$ is an irreducible cuspidal unitary representation of a general linear group and $\sigma$ is an irreducible cuspidal representation of a classical group.
\end{abstract}
\maketitle

\section*{Introduction}
The problem of determining the composition series of  induced representations is important for the representation theory.
Here we consider a certain class of representations induced from two segments and a cuspidal representation of a classical group. 
 We use the Langlands classification for irreducible representations of classical groups and 
the M{\oe}glin-Tadi\'c classification for discrete series. Our approach is based on tools of the Jacquet module and intertwining operators.

To describe our results we introduce some notation. 
Fix a local non-archimedean field $F$ of characteristic different than two.
Let $\rho$ be
an irreducible cuspidal unitary representation of some $GL(m,F)$, and $x,y\in\mathbb{R}$, such that $y-x+1\in\mathbb{Z}_{\geq 0}$. By Zelevinsky classification, the set 
$\Delta=[\nu^x \rho,\nu^y \rho]=\{\nu^x \rho,...,\nu^y \rho\}$ is called a segment.
We have a unique irreducible subrepresentation
\[
\delta(\Delta)=\ds{x}{y}\hookrightarrow \nu^y\rho\times \cdots \times \nu^x\rho,
\]
of the parabolically induced representation.
If $\Delta\subseteq \Delta'$, then 
$\delta(\Delta)\times \delta(\Delta')
\cong
\delta(\Delta')\times \delta(\Delta)
$ is irreducible.
Set 
$e(\Delta)=(x+y)/2$.
Given a sequence of segments $\Delta_1,...,\Delta_k$, such that 
$e(\Delta_1)\geq \cdots \geq e(\Delta_k)>0$ and
 an irreducible tempered representation $\tau$, of a symplectic or (full) orthogonal group, we have a unique quotient,
 called the Langlands quotient,
\[
\delta(\Delta_{1})\times\cdots\times \delta(\Delta_k)\rtimes \tau
\rightarrow
L(\delta(\Delta_{1})\times\cdots\times \delta(\Delta_k)\rtimes \tau),
\]
of the parabolically induced representation. Similarly to $\tau$, assume that $\sigma$ is cuspidal such that $\nu^\frac{1}{2}\rho \rtimes \sigma$ reduces.
Let $a, b, c \in \mathbb{Z}+\frac{1}{2}$ such that $\frac{1}{2}\leq a < b  < c$. We are considered with composition series of induced representation
\[
\ds{-a}{c}\times \ds{\frac{1}{2}}{b}\rtimes \sigma.
\]
Now, we introduce some discrete series, appearing as only irreducible subrepresentations in the following formulas.
\[
\sigma_a \hookrightarrow \ds{\frac{1}{2}}{a}\rtimes \sigma, 
\textrm{ and similarly for  } \sigma_b \textrm{ and  } \sigma_c. \textrm{ Further }
\]
\begin{align*}
    \sigma_{b,c}^+
\hookrightarrow  
\ds{\frac{1}{2}}{b}\rtimes\sigma_{c},
\quad
 \sigma_{b,c}^++
\sigma_{b,c}^-
\hookrightarrow 
\ds{-b}{c}\rtimes \sigma,
\end{align*}
and similarly for $\sigma_{a,c}^\pm$. Finally
\(
\sigma_{b,c,a}^\pm
\hookrightarrow
\ds{\frac{1}{2}}{a} \rtimes \sigma_{b,c}^\pm
\). Now we have
\begin{theorem*} 
Let $\psi=\delta([\nu^{-a}\rho,\nu^c\rho])\times \delta([\nu^\frac{1}{2}\rho,\nu^b\rho])\rtimes \sigma$ and define representations
\begin{align*}
    W_1=&\sigma_{b,c,a}^+ +L(\ds{\frac{1}{2}}{a}\rtimes \sigma_{b,c}^-),
        \\
    W_2=&L(\ds{\frac{1}{2}}{a}\rtimes \sigma_{b,c}^+)
        +
        L(\ds{-a}{b}\rtimes \sigma_c)
        +
        \\
        &L(\ds{\frac{1}{2}}{b}\rtimes \sigma_{a,c}^-)
        +
        L(\ds{-b}{c}\rtimes \sigma_a),
        \\
     W_3=&L(\ds{\frac{1}{2}}{b} \rtimes \sigma_{a,c}^+)+
            L(\ds{-a}{c}\rtimes\sigma_{b})+
        \\
          &\sigma_{b,c,a}^- + 
                            L(\ds{-b}{c}\times
                                \ds{\frac{1}{2}}{a} 
                                    \rtimes \sigma),
        \\
      W_4=&L(\psi).                      
\end{align*}
Then there exists a sequence $\{0\}=V_0\subseteq V_1
\subseteq V_2 \subseteq V_3
\subseteq V_4=\psi$,
such that
\begin{equation*}
V_i/V_{i-1}\cong W_i,\quad  i=1,\ldots,4.
\end{equation*}    
Further, $W_1$ is chosen to be the largest possible, then $W_2$, and so on.
\end{theorem*}
Now we describe the content of the paper.
After Preliminaries, we fix the notation in Section
\ref{sect:2} and collect some reducibility results.
Intertwining operators and an approach to decompose the induced representation are considered in Section \ref{dekompozicija1}. In Section \ref{diskretni}, we determine the occuring discrete series. The 
remaining non-tempered candidates, not provided in 
Section \ref{dekompozicija1},
are  listed in Section 
\ref{netemperirani}. Their occurrence is confirmed in Sections \ref{mult_L1}-\ref{mult_L3}. Composition factors are described in Section \ref{kompozicijski_faktori}. To determine composition series, we decompose kernels of intertwining operators in Sections \ref{kompozicijski1}-\ref{kompozicijski3},
and provide the main result in Section \ref{dekompozicija2}.
\ \\ 

The author would like to thank Ivan Mati\'c for pointing to a result important for this paper.

\section{Preliminaries}
\label{sect:1}
Let $F$ be a local non-archimedean field of characteristic different than two. As in \cite{tadic-diskretne}, fix a tower of symplectic or orthogonal non-degenerate $F$ vector spaces $V_n$, $n\geq 0$ where $n$ is the Witt index. We denote by $G_n$ the group of isometries of $V_n$. It has split rank $n$.  
Also, we fix the set of standard parabolic subgroups in the usual way.  
Standard parabolic proper subgroups of $G_n$ are in bijection
with the set of ordered  partitions of positive integers 
$m\leq n$: 
\begin{align*}
&\{s=(n_1,\ldots,n_k) \mid  n_1+\cdots +n_k=m,  k>0 
\} \longleftrightarrow P_s,
\\
&P_s=M_s N_s, 
\quad
\textrm{Levi factorization with }
M_s \textrm{ Levi factor}, 
\\
&M_s
\cong GL(n_1,F)\times \cdots \times GL(n_k,F)  \times G_{n-m}.
\end{align*}
By Alg $G_n$ we denote smooth representations of $G_n$,
Irr $G_n$ irreducible representations, 
and
subscript $f.l.$ means finite length, $u$ unitary, and $cusp$ cuspidal. Also denote 
\(\textrm{Alg } G=\cup_{n\geq 0}  \textrm{Alg } G_n\),
and so on.
We use a similar notation for $GL(n,F)$.
For $\delta_i \in \textrm{Alg  }GL(n_i,F)$, $i=1,...,k$ and 
$\tau \in \textrm{Alg  }G_{n-m}$, 
let
$\pi=\delta_1\otimes\cdots\otimes\delta_k\otimes \tau$ 
$\in \textrm{Alg  } M_s$ and
\[
\delta_1\times\cdots\times\delta_k\rtimes \tau= 
\text{Ind}_{M_s}^{G_n}( \pi)
\]
be the representation induced  from $\pi$ using normalized parabolic induction.
If $\sigma \in \textrm{Alg }G_n$ we denote by 
$\text{r}_{s}(\sigma)=\text{r}_{M_s}(\sigma)=\text{Jacq}_{M_s}^{G_n}(\sigma)$
the normalized Jacquet module of $\sigma$. We have the Frobenius reciprocity
\[
\text{Hom}_{G_n}(\sigma,\text{Ind}_{M_s}^{G_n}(\pi))=
\text{Hom}_{M_s}(\text{Jacq}_{M_s}^{G_n}(\sigma),\pi).
\]
Let $\rho \in  \textrm{Irr}_{u,cusp} GL$ 
and $x,y\in\mathbb{R}$, such that $y-x+1\in\mathbb{Z}_{\geq 0}$. The set 
\begin{align*}
\Delta=[\nu^x \rho,\nu^y \rho]=\{\nu^x \rho,...,\nu^y \rho\}
\end{align*}
is called a segment.
We have a unique irreducible subrepresentation
\[
\ds{x}{y} \hookrightarrow \nu^y\rho\times \cdots \times \nu^x\rho,
\]
of the induced representation, and it is essentially square integrable.
We also denote 
$e([\nu^x \rho,\nu^y \rho])=e(\delta([\nu^x \rho,\nu^y \rho])=\frac{x+y}{2}$.
 For $y-x+1\in\mathbb{Z}_{< 0}$ define $[\nu^x \rho,\nu^y \rho]=\emptyset$ and $\delta(\emptyset)$ is the irreducible representation of the trivial group.
Let
$\widetilde{\Delta}=[\nu^{-y} \widetilde{\rho},\nu^{-x} \widetilde{\rho}]$ where $\widetilde{\rho }$ denotes the contragredient of $\rho$. We have 
$\delta(\Delta)\widetilde{\ }=\delta(\widetilde{\Delta})$.
 By \cite{zelevinsky:ind-repns-II} if $\delta \in \textrm{Irr } GL$
is essentially square integrable,
there exists a segment $\Delta$ such that $\delta=\delta(\Delta)$.
Let
$\delta_i=\delta( \Delta_i), e_i=e(\delta_i),i=1,2$. We have
\[
    \delta_1\times \delta_2 \textrm{ reduces}
    \Leftrightarrow 
    \Delta_1 \cup \Delta_2 \textrm{ is a segment and } 
    \Delta_1 \nsubseteq \Delta_2,
    \Delta_2 \nsubseteq \Delta_1.
\] 
In that case, if $e_1\geq e_2$, the induced representation has a unique irreducible quotient, called Langlands quotient, 
and a 
unique irreducible subrepresentation. They swap positions in $\delta_2\times \delta_1$ and make composition factors. We have an exact sequence.
\[
\delta(\Delta_1 \cup \Delta_2)
\times
\delta(\Delta_1 \cap \Delta_2)
\rightarrow 
\delta_1 \times\delta_2
\rightarrow L(\delta_1 \times\delta_2)
=L(\delta_1,\delta_2).
\]
Given a sequence  
$\delta_i=(\Delta_i)$,  $i=1,\ldots,k$ such that 
$e(\Delta_1)\geq \cdots \geq e(\Delta_k)>0$
and  
$\tau \in \textrm{Irr }G$, tempered, the Langlands quotient is a unique irreducible quotient:
\[
\delta_1\times\cdots\times \delta_k\rtimes \tau
\rightarrow
L(\delta_1\times\cdots\times \delta_k\rtimes \tau),
\]
and it appears with multiplicity one in the induced  representation.
It is also a unique irreducible subrepresentation of 
$\widetilde{\delta_1}\times\cdots\times \widetilde{\delta_k}\rtimes \tau
\cong
(\delta_1\times\cdots\times \delta_k\rtimes \tau)
^{\widetilde{\ }}
$.
Permuting $\delta_i$-s and possibly taking contragredients does not change composition factors.
Every irreducible representation of $G_n$ can be written as a Langlands quotient.

If $\sigma$ is a discrete series representation of $G_n$ then by the
M{\oe}glin-Tadi\'c, now unconditional, classification  (\cite{moeglin},\cite{tadic-diskretne}),  
it is described by an admissible triple
\[
(\text{Jord},\sigma_{cusp},\epsilon).
\]
Here 
$\textrm{Jord}$ is a set of pairs $=(a,\rho)$ where 
$\widetilde{\rho}\cong\rho \in \textrm{Irr}_{u,cusp} GL$ and
$a\in \mathbb{Z}_{>0}$, of parity depending on $\rho$, such that 
$\delta([\nu^{-(a-1)/2}\rho,\nu^{(a-1)/2}\rho])\rtimes \sigma$ 
is irreducible, but reduces for some larger integer $a'$. 
We write $\text{Jord}_{\rho}=\{ a : (a,\rho)\in \text{Jord} \}$ and for
 $a\in \text{Jord}_{\rho}$ let
 $a_ {-}$  be the largest element of $\text{Jord}_{\rho}$ strictly less than $a$, if such exists. Next, there exists a unique, up to an isomorphism,  
 $\sigma_{cusp} \in \textrm{Irr}_{cusp} G$, such that 
there exists  
$\pi \in \textrm{Irr} GL $ 
and
$\sigma \hookrightarrow \pi \rtimes \sigma_{cusp}$. It is called the partial cuspidal support of $\sigma$.
 Finally, 
 $\epsilon$ is  a function from a subset of $\text{Jord} \cup (\text{Jord}\times \text{Jord})$ into  $\{\pm1\}$. 
Assume $(a,\rho)\in \text{Jord}$ and $a$ is even. 
Then $\epsilon(a,\rho)$ is defined, and  
if $a=\text{min}(\text{Jord}_{\rho})$
\[
\epsilon(a,\rho)=1
\Leftrightarrow  
\exists
\pi' \in \textrm{Irr}G, 
\quad
\sigma \hookrightarrow \delta([\nu^{1/2} \rho,\nu^{(a-1)/2} \rho])  \rtimes \pi'
,
\]
while if $a_{-}$ exists
\[
\epsilon(a,\rho)\epsilon (a_{-},\rho)^{-1}=1
\Leftrightarrow 
\exists \pi'' \in \textrm{Irr}G,
 \quad  \sigma \hookrightarrow 
   \delta([\nu^{(a_{-}+1)/2} \rho,\nu^{(a-1)/2} \rho])  \rtimes \pi''
.
\]

 Now we recall the Tadi\'c formula for computing Jacquet modules. Let $R(G_n)$ be the Grothendieck group of the category of smooth representations of $G_n$ of finite length. It is the free Abelian group generated by classes of irreducible representations of $G_n$.
 If $\sigma$ is a smooth finite length representation of $G_n$ denote by $\text{s.s.}(\sigma)$ the semisimplification of $\sigma$, that is the sum of classes of composition series of $\sigma$.
Put $R(G)=\oplus_{n\geq 0} R(G_n)$.
Let $R^+_0(G)$ be a $\mathbb{Z}_{\geq 0}$
subspan of
classes of irreducible representations.
For $\pi\in R(G)$ we define
$\lfloor \pi \rfloor_{R^+_0(G)}\in R^+_0(G)$ such that
$\lfloor \pi \rfloor_{R^+_0(G)} -\pi\in R^+_0(G)$. For $\pi_1,\pi_2\in R(G)$ we define $\pi_1 \leq \pi_2$ if  $\pi_2-\pi_1 \in R^+_0(G)$. 
 Similarly define $R(GL)=\oplus_{n\geq 0}R(GL(n,F))$. We have the map 
 $\mu^* : R(G)\rightarrow R(GL) \otimes R(G)$ defined by
 \[
 \mu^*(\sigma)=1\otimes \sigma + \sum_{k=1}^{n} \text{s.s.}(r_{(k)}(\sigma)),\  \sigma \in R(G_n).
 \]
The following result derives from Theorems 5.4 and 6.5 of  \cite{tadic-structure}, see also section 1.\ in \cite{tadic-diskretne}. They are based on Geometrical  Lemma (2.11 of \cite{bernstein-zelevinsky:ind-repns-I}).
\begin{theorem} Let 
$\sigma \in \textrm{Alg}_{f.l.}G$, 
and
$[\nu^x\rho,\nu^y\rho]\neq \emptyset$ a segment.
Then
\begin{equation} \label{komnozenje}
\begin{split}
\mu^*(&\ds{x}{y}\rtimes \sigma)=
\sum_{\delta'\otimes\sigma'\leq \mu^*(\sigma)}
\sum_{i=0}^{y-x+1} \sum_{j=0}^{i}                                \\
&\delta([\nu^{i-y}\widetilde{\rho},\nu^{-x}\widetilde{\rho}])
\times
\delta([\nu^{y+1-j}\rho,\nu^{y}\rho])\times \delta'
\otimes
\delta([\nu^{y+1-i}\rho,\nu^{y-j}\rho])\rtimes \sigma'
\end{split}
\end{equation}
where $\delta'\otimes\sigma'$ denotes an irreducible subquotient in the appropriate Jacquet module.
\end{theorem}

Now we
provide results of \cite{segment}, about Jacquet modules. 
Consider induced representation
\[
\delta([\nu^{-a}\rho,\nu^c \rho])\rtimes \sigma
\]
where  $\sigma \in \textrm{Irr}_{cusp} G$, 
$\rho \in \textrm{Irr}_{u,cusp}GL$, such that 
$\nu^\frac{1}{2}\rtimes \sigma$ reduces, and $a,c\in \mathbb{Z}+\frac{1}{2}$,  such that $c-a\geq 0$. 
We use notation, that we shall explain, such that in  $R(G)$:
\[
\delta([\nu^{-a}\rho,\nu^c \rho])\rtimes \sigma
=
\delta([\nu^{-a}\rho,\nu^c \rho]_+;\sigma)
+
\delta([\nu^{-a}\rho,\nu^c \rho]_-;\sigma)
+
L(\delta([\nu^{-a}\rho,\nu^c \rho]);\sigma)
\]
where some expressions on the right hand side, are defined as zero, depending on $a$ and $c$. More precisely, on the right hand side we have in $R(G)$:
\begin{equation} 
\label{segment-jacquet-equation}
    \begin{split}
\frac{1}{2}< -a :
& \quad L(\ds{-a}{c};\sigma),
\\
-a=\frac{1}{2}:& \quad
\delta([\nu^{-a}\rho,\nu^c \rho]_+;\sigma) +
L(\delta([\nu^{-a}\rho,\nu^c \rho]);\sigma),
\\
\frac{1}{2} \leq a\neq c:& \quad
\delta([\nu^{-a}\rho,\nu^c \rho]_+;\sigma)
+\delta([\nu^{-a}\rho,\nu^c \rho]_-;\sigma),
\\
\frac{1}{2} \leq a\neq c: & \quad
\delta([\nu^{-a}\rho,\nu^c \rho]_+;\sigma)
+\delta([\nu^{-a}\rho,\nu^c \rho]_-;\sigma)
+L(\delta([\nu^{-a}\rho,\nu^c \rho]);\sigma).
\end{split}
\end{equation}
In the first case, the induced representation is irreducible. In the second, we have a discrete series subrepresentation and a Langlands quotient. In the third case, we have two tempered subrepresentations (see \cite{tadic:square-integr-correspond-segments}) of the induced representation, and we denote by 
$\delta([\nu^{-a}\rho,\nu^c \rho]_+;\sigma)$ the one that  
 has in its minimal standard Jacquet module at least one irreducible subquotient whose all exponents are non-negative. 
 In the last case, we have two discrete series subrepresentations of the induced representation and a Langlands quotient.
Now we have
\begin{equation} \label{jacquet-segment}
    \begin{split}
        \mu^*(\delta([\nu^{-a}\rho,\nu^c \rho]_\pm;&\sigma))=
        \sum_{i=-a-1}^{\pm\frac{1}{2}-1}
        \delta([\nu^{-i}\rho,\nu^a\rho])
        \times
        \delta([\nu^{i+1}\rho,\nu^c\rho])
        \otimes
        \sigma+
        \\
        \sum_{i=-a-1}^a\sum_{j=i+1}^c
        &\delta([\nu^{-i}\rho,\nu^a\rho])
        \times
        \delta([\nu^{j+1}\rho,\nu^c\rho])
        \otimes
        \delta([\nu^{i+1}\rho,\nu^j \rho]_\pm;\sigma)
        +
        \\
        \underset{\textrm{\ } \quad i+j<-1}{
            \sum_{-a-1\leq i\leq a}\textrm{\ }
            \sum_{ i+1\leq j\leq a}
            }
        &\delta([\nu^{-i}\rho,\nu^a\rho])
        \times
        \delta([\nu^{j+1}\rho,\nu^c\rho])
        \otimes
        L(\delta([\nu^{i+1}\rho,\nu^j \rho]);\sigma).
    \end{split}
\end{equation}
If we write 
$\delta([\nu^\frac{1}{2}\rho,\nu^{-\frac{1}{2}}\rho])
\otimes \sigma$ for $\sigma$, we have 
\begin{equation} \label{jacquet-strogo-pozitivna}
    \begin{split}
        \mu^*(\delta([\nu^\frac{1}{2}\rho,\nu^c \rho]_+;\sigma))=
        \sum_{j=-\frac{1}{2}}^c
        \delta([\nu^{j+1}\rho,\nu^c\rho])
        \otimes
        \delta([\nu^\frac{1}{2}\rho,\nu^j \rho]_+;\sigma).
    \end{split}
\end{equation}
And for $a<\frac{1}{2}$ or $\frac{1}{2}\leq a < c$
we have
\begin{equation} \label{jacquet-langlandsov-kvocijent}
    \begin{split}
\mu^*(L(\delta([\nu^{-a}\rho,\nu^c \rho]);& \sigma))=
        \sum_{i=\frac{1}{2}}^{c} 
        L(\delta([\nu^{-i}\rho,\nu^a\rho])
        ,
        \delta([\nu^{i+1}\rho,\nu^c\rho]))
        \otimes
        \sigma+
        \\
        \underset{\textrm{\ } \quad 0\leq i+j}{
            \sum_{-a-1\leq i\leq c}\textrm{\ }
            \sum_{ i+1\leq j\leq c}
            }
        & L(\delta([\nu^{-i}\rho,\nu^a\rho]),
        \delta([\nu^{j+1}\rho,\nu^c\rho]) )
        \otimes
        L(\delta([\nu^{i+1}\rho,\nu^j \rho]);\sigma).
    \end{split}
\end{equation}

\section{Notation and basic reducibilities}
\label{sect:2}
In this section, we fix the notation and prepare some reducibility results. 
Let $\rho$ be an irreducible unitary cuspidal representation  of $GL(m_{\rho},F)$ and  
$\sigma$ an irreducible cuspidal representation of $G_n$ such that $\nu^\frac{1}{2}\rho \rtimes \sigma$ reduces. By Proposition 2.4 of
\cite{tadic:red-par-ind} $\rho$ is self-dual. 
We consider 
\[
\frac{1}{2} \leq a, b, c \in \mathbb{Z}+\frac{1}{2},
\]
that need not be fixed, but when appearing together in a formula, we have, depending on which appears, 
$ a < b  < c$. We denote
the representation we want to decompose
\[
\psi=\ds{-a}{c}\times \ds{\frac{1}{2}}{b}\rtimes \sigma.
\]
 Further, we shorten some  notations from 
 \eqref{segment-jacquet-equation}:
\begin{align*}
\sigma_a=
\delta([\nu^{\frac{1}{2}}\rho,\nu^a\rho]_+;\sigma),
\quad
\sigma^-_{b,c}=
\delta([\nu^{-b}\rho,\nu^c \rho]_-;\sigma)
,
\quad
\sigma^+_{b,c}=
\delta([\nu^{-b}\rho,\nu^c \rho]_+;\sigma).    
\end{align*}
The following result is Theorem 2.3 from 
\cite{muic:composition series}.
\begin{theorem}
With discrete series being subrepresentations, we have in $R(G)$
\label{muic-diskretne-podreprezentacije} 
\label{prva}
\begin{align*}
\delta([\nu^\frac{1}{2}\rho,\nu^a\rho])\rtimes \sigma
&=
\sigma_a+L(\delta([\nu^\frac{1}{2}\rho,\nu^a\rho])\rtimes \sigma),
\\
\delta([\nu^{-b}\rho,\nu^c\rho])\rtimes \sigma
&=
 \sigma_{b,c}^++
\sigma_{b,c}^-+
L(\delta([\nu^{-b}\rho,\nu^c\rho])\rtimes \sigma).
\end{align*}
Here
\begin{align*}
\text{Jord}(\sigma_a)
&=\{(2a+1,\rho)\}\cup Jord(\sigma),
\\
Jord(\sigma^+_{b,c})=Jord(\sigma^-_{b,c})
&=
\{(2b+1,\rho),(2c+1,\rho)\}\cup Jord(\sigma).
\end{align*}
Further,
$\epsilon_{\sigma_a},
\epsilon_{\sigma_{b,c}^+}$, 
and
$\epsilon_{\sigma_{b,c}^-}  \textrm{ extend } \epsilon_{\sigma}$,
such that
$\epsilon_{\sigma_a}(2a+1,\rho)=1$, and
\\
\(\epsilon_{\sigma^+_{b,c}}(2b+1,\rho)=
\epsilon_{\sigma^+_{b,c}}(2c+1,\rho)=1,
 \epsilon_{\sigma^-_{b,c}}(2b+1,\rho)=
 \epsilon_{\sigma^-_{b,c}}(2c+1,\rho)=-1.\)
\end{theorem}
The next proposition follows from Theorem 2.1 of \cite{muic:composition series}.
\begin{proposition}
\label{druga}  
We use $\sigma_{b,c,a}^+=\sigma_{a,b,c}^+$, $\sigma_{b,c,a}^-$ and $\sigma_{a,b,c}^-$
to denote nonisomorphic discrete series, such that in $R(G)$ we have
\begin{equation*}
\begin{split}
&\delta([\nu^{-b}\rho,\nu^c\rho])\rtimes  \sigma_a=
\sigma_{b,c,a}^+ +
\sigma_{b,c,a}^- +
L(\delta([\nu^{-b}\rho,\nu^c\rho])\rtimes \sigma_a)\quad \textrm{and}
\\
&\ds{-a}{b}\rtimes  \sigma_c=
\sigma_{b,c,a}^+ +
\sigma_{a,b,c}^- +
L(\ds{-a}{b}\rtimes  \sigma_c).
\end{split}    
\end{equation*}
These discrete series appear as subrepresentations in induced representations. Also 
\begin{equation*}
\begin{split}
\textrm{Jord}(\sigma_{b,c,a}^+)=
&\textrm{Jord}(\sigma_{b,c,a}^-)=
\textrm{Jord}(\sigma_{a,b,c}^-)=
\\
&\{(2a+1,\rho),(2b+1,\rho),(2c+1,\rho)\}
\cup Jord(\sigma)
\end{split}
\end{equation*}
\\
and $\epsilon_{\sigma_{b,c,a}^+}$, $\epsilon_{\sigma_{b,c,a}^-}$ and 
$\epsilon_{\sigma_{a,b,c}^-}$
extend $\epsilon_{\sigma}$ such that
\begin{align*}
\epsilon_{\sigma_{b,c,a}^+}(2a+1,\rho)&=1, \quad
& \epsilon_{\sigma_{b,c,a}^+}(2b+1,\rho)&=1,  \quad
 &\epsilon_{\sigma_{b,c,a}^+}(2c+1,\rho)&=1, \quad
 \\
  \epsilon_{\sigma_{b,c,a}^-}(2a+1,\rho)&=1, \quad
 &\epsilon_{\sigma_{b,c,a}^-}(2b+1,\rho)&=-1, \quad
 &\epsilon_{\sigma_{b,c,a}^-}(2c+1,\rho)&=-1,  \quad
 \\
 \epsilon_{\sigma_{a,b,c}^-}(2a+1,\rho)&=-1, \quad
 &\epsilon_{\sigma_{a,b,c}^-}(2b+1,\rho)&=-1, \quad
 &\epsilon_{\sigma_{a,b,c}^-}(2c+1,\rho)&=1. \quad
 \end{align*}
\end{proposition}
Observe that
\begin{equation} \label{jacq-diskretne}
\begin{split}
    \mu^*(\sigma^+_{b,c,a}) \geq &
        \ds{-a}{b}\otimes \sigma_c +  
        \ds{-b}{c}\otimes \sigma_a,
\\
\mu^*(\sigma^-_{b,c,a}) \geq &
          \ds{-b}{c}\otimes \sigma_a,
\\
\mu^*(\sigma^-_{a,b,c}) \geq &
        \ds{-a}{b}\otimes \sigma_c. 
\end{split}
\end{equation}
We finished introducing notation and state some more reducibility results.

{\ }

Here is a consequence of
Theorem 6.3 of \cite{tadic:regular-square}, see also section 3 there.
\begin{proposition} 
\label{multiplicitet-stroge}
We have in $R(G)$, with multiplicity one:
\[
\nu^a\rho\times\cdots \times \nu^{\frac{1}{2}}\rtimes \sigma \geq \sigma_a. 
\]
\end{proposition}

The next lemma follows from Theorem 5.1 of \cite{muic:composition series}, ii).
\begin{lemma} 
\label{lema-diskretna-podreprezentacija}
    We have in $R(G)$,
\begin{equation*}
    \begin{split} 
    \ds{\frac{1}{2}}{a}\rtimes
      \sigma_{b}
    =&
    \sigma_{a,b}^+
    +
    \textrm{L}(\ds{\frac{1}{2}}{a}\rtimes
      \sigma_{b}), \textrm{ and }
\\
      \ds{\frac{1}{2}}{b}\rtimes
       \sigma_a
     =&
      \sigma_{a,b}^+
     +
      \textrm{L}(\ds{-a}{b}
       \rtimes \sigma)+
\\     
     &
      \textrm{L}(
       \ds{\frac{1}{2}}{a}\rtimes
        \sigma_b)
      +
      \textrm{L}(
       \ds{\frac{1}{2}}{b}\rtimes
        \sigma_a).
\end{split}
\end{equation*}
\end{lemma}
By Proposition 2.4 of  \cite{ciganovic} we have
\begin{lemma} 
\label{lema-pola-a-bc-tvrdnja} 
We have in $R(G)$
    \begin{align} 
    \label{lema-pola-a-bc-f}
    \ds{\frac{1}{2}}{a} \rtimes \sigma_{b,c}^\pm
    &=
    \sigma_{b,c,a}^\pm
    +
    L(\ds{\frac{1}{2}}{a} \rtimes \sigma_{b,c}^\pm ),
\textrm{ so }
\\
\label{lema-pola-a-bc-f1}
\mu^*(\sigma_{b,c,a}^\pm)
&\geq 
\ds{\frac{1}{2}}{a} \otimes \sigma_{b,c}^\pm.
    \end{align}
\end{lemma}

The next is Proposition 3.2 of 
\cite{faktor}.
\begin{theorem} 
\label{muic-A3}
With discrete series being a subrepresentation, we have in $R(G)$
\begin{align*}
    \delta([\nu^{-a}\rho,\nu^c\rho])
      \rtimes \sigma_b&=
      \sigma_{b,c,a}^+ +
            L(\delta([\nu^{-a}\rho,\nu^c\rho])\rtimes \sigma_b)+
        \\    
      &L(\delta([\nu^{-a}\rho,\nu^b\rho])\rtimes \sigma_c)  
        +
      L(\delta([\nu^{-b}\rho,\nu^c\rho])\rtimes \sigma_a).
\end{align*}
\end{theorem}

Finally we have the main result of \cite{ciganovic}.
\begin{theorem} \label{ciganovic-A2} 
With discrete series being a subrepresentation, we have in $R(G)$
\begin{equation*}
\begin{split}
\delta([\nu^{-b}\rho,\nu^c\rho]) \times 
\delta([\nu^\frac{1}{2}\rho,\nu^a\rho])\rtimes \sigma
=
L(\delta([\nu^\frac{1}{2}\rho,\nu^a\rho])\rtimes \sigma_{b,c}^+)
+
L(\delta([\nu^\frac{1}{2}\rho,\nu^a\rho])\rtimes \sigma_{b,c}^-)
\\
+\sigma_{b,c,a}^+ 
+
\sigma_{b,c,a}^-  
+
L(\delta([\nu^{-b}\rho,\nu^c\rho])\rtimes \sigma_a)
+
L(\delta([\nu^{-b}\rho,\nu^c\rho])\times \delta([\nu^\frac{1}{2}\rho,\nu^a\rho])\rtimes \sigma).
\end{split}
\end{equation*}
\end{theorem}


\section{Decomposing mixed case} 
\label{dekompozicija1}

As we are interested in the composition series of induced representations, we shall need a result that follows  from proofs of Theorems 2-1 and 2-6 from  \cite{muic-hanzer:langzel}.
\begin{theorem}
\label{contra}
There exists a contravariant exact functor:
\[
    \textrm{Alg  }G_{n} \overset{\wedge}{\longrightarrow} \textrm{Alg  }G_{n},
\]
such that
\[
    \overset{\wedge}{\pi} \cong \pi, \quad \pi \in \textrm{Irr } G_n,
\]
and if  $\delta_i \in \textrm{Irr  }GL(n_i,F)$, $i=1,...,k$, $m=n_1+\cdots+n_k$ and 
$\tau \in \textrm{Irr  }G_{n-m}$ we have
\[
(\delta_1\times\cdots\times\delta_k\rtimes \tau)^\wedge
\cong
\widetilde{\delta_1}\times\cdots\times
\widetilde{\delta_k}\rtimes \tau.
\]
\end{theorem}
\begin{proof}
    We follow the same lines as in proofs of Theorems 2-1 and 2-6 from  
    \cite{muic-hanzer:langzel}. 
    If $G_n$ is orthogonal, then $\wedge$ is just  contragredient. If $G_n$ is symplectic, 
    by a result of Waldspurger 
    (\cite{M-V-W} Chapter 4, II.1), for any element 
    $\eta \in GSp(2n)$ of similitude $-1$, and $\pi\in Irr G_n$ we have
    $\widetilde{\pi} \cong \pi^\eta$, where 
    $\pi^\eta(g)=\pi(\eta g \eta^{-1})$.
    Now 
     choose an element of the form 
    $\eta = (id, \eta')\in GL(n, F) \times GSp(0, F)=
    GL(n, F) \times F^\times
    $, 
    identified with the Levi subgroup of the appropriate maximal parabolic
    subgroup of $GSp(2n, F)$, where $\eta'$ 
    is an element with similitude equal to $-1$, as $\eta$ is.
    For $\wedge =  \eta \circ \sim$, we have
    \[
    (\delta_1\times\cdots\times\delta_k\rtimes \tau)^{ \sim \eta}
    \cong
    (\widetilde{\delta_1}\times\cdots\times
    \widetilde{\delta_k}\rtimes \widetilde{\tau})^\eta
    \cong
    \widetilde{\delta_1}\times\cdots\times
    \widetilde{\delta_k}\rtimes \tau.
    \]
\end{proof}
Consider some standard intertwining operators
\begin{equation*}
\begin{tikzcd}
\A \arrow[r,"\cong"] 
\arrow[d,"f_0","\cong" '] & \F \arrow[d,"g_0","\cong" '] \\ 
\B \arrow[d,"f_1"] &        \G \arrow[d,"g_1"] \\ 
\C \arrow[d,"f_2"] &        \I \arrow[d,"g_2"] \\ 
\D \arrow[d,"f_3"] &        \J \arrow[d,"g_3"] \\ 
\E \arrow[r,"\cong"] &    \K 
\end{tikzcd}
\end{equation*}
We denoted
$\psi=\ds{-a}{c}\times \ds{\frac{1}{2}}{b}\rtimes \sigma$. Also
for all $i\geq 1$ denote
$K_i=\textrm{Ker }f_i$, and $H_i=\textrm{Ker } g_i$.
By Theorems \ref{prva} and \ref{contra}, we have 
\begin{align*}
K_1&\cong {H_3}^\wedge\cong 
\delta([\nu^\frac{1}{2}\rho,\nu^b\rho])\rtimes
      \sigma_{a,c}^+
      +
      \delta([\nu^\frac{1}{2}\rho,\nu^b\rho])\rtimes
      \sigma_{a,c}^-,
 \\
 K_2&\cong H_2^\wedge\cong
 \delta([\nu^{-c}\rho,\nu^b\rho])\times 
      \delta([\nu^{\frac{1}{2}}\rho,\nu^a\rho])\rtimes \sigma,
 \\
 K_3&\cong H_1^\wedge \cong
 \delta([\nu^{-c}\rho,\nu^a\rho])
      \rtimes \sigma_b,  
\end{align*}
and no kernel contains $L(\psi)$. Thus we have:
\[
Im(f_3\circ \cdots \circ f_0)
\cong
L(\psi)
\cong
Im(g_3\circ \cdots \circ g_0),
\]
and the diagram is commutative up to a constant. We have in $R(G)$:
\begin{equation}
    \label{long-intertwining}
\forall i\quad   K_i \leq  \psi \leq K_1+K_2+K_3 +L(\psi).
\end{equation}
The composition factors of $K_2$ and $K_3$ are determined  by Theorems \ref{muic-A3} and 
\ref{ciganovic-A2}. So in search of remaining subquotients, we need to 
decompose $K_1$. After that, 
we determine the multiplicities of all subquotients of $\psi$.
Note that by Lemma 8.1 of \cite{tadic-diskretne}, all tempered subquotients of $\psi$
are discrete series.


\section{Discrete series subquotients}
\label{diskretni}
Here we determine discrete series subquotients in three induced representations
\begin{equation} \label{A-A1}
\begin{split}
\delta([\nu^{-a}\rho,\nu^c\rho])\times \delta([\nu^\frac{1}{2}\rho,\nu^b\rho])\rtimes \sigma 
\geq &
\\
      \delta([\nu^\frac{1}{2}\rho,\nu^b\rho])\rtimes
      \sigma_{a,c}^+
      &+
      \delta([\nu^\frac{1}{2}\rho,\nu^b\rho])\rtimes
      \sigma_{a,c}^-,
\end{split}
\end{equation}
where the inequality follows from Theorem \ref{prva}.
We start with candidates.
\begin{lemma} \label{jacquet-discrete-A}
Only possible discrete series subquotients appearing in
\noindent
\[
\delta([\nu^{-a}\rho,\nu^c\rho])\times \delta([\nu^\frac{1}{2}\rho,\nu^b\rho])\rtimes \sigma 
\]
are 
$\sigma^+_{b,c,a}$, $\sigma^-_{b,c,a}$ and $\sigma^-_{a,b,c}$.
\end{lemma}
\begin{proof}
Consider cuspidal support of the induced representation and M{\oe}glin Tadić classification of discrete series.
Since $\nu^{\pm \frac{1}{2}}\rho$ appears 3 times in the cuspidal support, possible discrete series are subrepresentations of a representation of the form
\[
\ds{-y}{z}\rtimes \sigma_x,
\]
where either $0<x<y<z$ or $0<y<z<x$. We look at the first case. Here $z$ is the largest such that $\nu^{\pm z}\rho$ appears once in the cuspidal support, so we must have $z=c$. Further, $y$ is the largest such that 
$\nu^{\pm y}\rho$ appears with times in the cuspidal support, so it must be $b$. Finally, $x$
is the largest such that 
$\nu^{\pm x}\rho$ appears three times in the cuspidal support, so it must be $a$. Same reasoning goes for the second case, where we have
$x=c$, $z=b$, and $y=a$. By Theorem 2.1 of \cite{muic:composition series}, we look for subrepresentations
\[
\sigma^+_{b,c,a}\oplus \sigma^-_{a,b,c}
\hookrightarrow \ds{-a}{b}\rtimes \sigma_c
\quad \textrm{and} \quad 
\sigma^+_{b,c,a}\oplus \sigma^-_{b,c,a}
\hookrightarrow \ds{-b}{c}\rtimes \sigma_a,
\]
\end{proof}

To determine which of these discrete series do appear in
\eqref{A-A1},
and what are their multiplicities, we need a couple of lemmas.

\begin{lemma} 
\label{lemma-diskr-u-dva-seg-kusp}
We have in $R(G)$, with maximum multiplicity
    \begin{equation}
    \label{lemma-diskr-u-dva-seg-kusp-f1}
\ds{\frac{1}{2}}{b}
    \times
    \ds{\frac{1}{2}}{c}\rtimes \sigma
  \geq \sigma_{b,c}^+ + \sigma_{b,c}^-
  +L(\ds{-b}{c}\rtimes \sigma).
    \end{equation}
\end{lemma}
\begin{proof}
Check multiplicity two of 
$\ds{-b}{c}\otimes \sigma$
and one of
$\ds{-c}{b}\otimes \sigma$
in
$\mu^*(\ds{\frac{1}{2}}{b}
    \times
    \ds{\frac{1}{2}}{c}\rtimes \sigma)
$ and use
Theorem \ref{muic-diskretne-podreprezentacije}.
\end{proof}

\begin{lemma} 
\label{dis-lema-jac-prva}
We have in $R(G)$, with maximum multiplicities:
\begin{align*}
\mu^*(
\delta([\nu^\frac{1}{2}\rho,\nu^b\rho])
\times &
\delta([\nu^{-a}\rho,\nu^c\rho]) 
 \rtimes \sigma)
\geq 
\\
&1 \cdot \ds{\frac{1}{2}}{a}\otimes \sigma_{b,c}^+
+
1 \cdot \ds{\frac{1}{2}}{a}\otimes \sigma_{b,c}^-.
\end{align*}
\end{lemma}

\begin{proof}
By \eqref{komnozenje}, we consider
$0\leq s \leq r \leq b+\frac{1}{2}$, 
$0\leq v \leq u \leq a+c+1$ and 
\begin{equation} 
\label{dis-lema-jac-prva-f1}
\begin{split}
    \ds{r-b}{-\frac{1}{2}} \times
    \ds{b+1-s}{b} \times
    \ds{u-c}{a} \times
    \ds{c+1-v}{c}&
\\ 
\otimes \ds{b+1-r}{b-s}
\times \ds{c+1-u}{c-v}\rtimes \sigma&.
\end{split}
\end{equation}
Searching for 
$\ds{\frac{1}{2}}{a}\otimes \sigma_{b,c}^\pm$,
we have 
$r=b+\frac{1}{2}$, and $s=v=0$.
So $u=c+\frac{1}{2}$ and 
use Lemma \ref{lemma-diskr-u-dva-seg-kusp}
on obtained
$
\sigma_{b,c}^\pm
\leq 
\ds{\frac{1}{2}}{b}
    \times
    \ds{\frac{1}{2}}{c}\rtimes \sigma$.

\end{proof}

\begin{lemma} \label{jacquet-diskr-A1+}
We have in $R(G)$, with maximum multiplicities:
\begin{align*}
\mu^*
(\delta([\nu^\frac{1}{2}\rho,\nu^b\rho])\rtimes
\sigma_{a,c}^+)
\geq
 &1 \cdot \ds{-a}{b}\otimes \sigma_c
\\
+&1
\cdot
\ds{\frac{1}{2}}{a}\otimes \sigma_{b,c}^+
+
0 \cdot
\ds{\frac{1}{2}}{a}\otimes \sigma_{b,c}^-.
\end{align*}
\end{lemma}
\begin{proof}
By \eqref{komnozenje} 
 consider
$0\leq s \leq r \leq b+\frac{1}{2}$,
$\delta'\otimes \sigma'\leq \mu^*(\sigma_{a,c}^+)$, and
\begin{equation} \label{formula-jac-pola-b-a-c}
\begin{split}
\ds{r-b}{-\frac{1}{2}}
    \times \ds{b-s+1}{b}
    \times
     \delta'
    \otimes 
    \ds{b+1-r}{b-s}
    \rtimes \sigma'.
\end{split}
\end{equation}
First we look for $\ds{-a}{b}\otimes \sigma_c$. Observe that 
$\nu^c\rho$ is not in a cuspidal support of $\delta'$, so in
\eqref{jacquet-segment} we have $j=c$ and
\[
\delta'\otimes\sigma'
\leq
\sum_{i=-a-1}^{a}
        \ds{-i}{a} 
    \otimes
        \delta(
            [\nu^{i+1}\rho,\nu^c\rho]_+;\sigma
        ).
\]
Searching for $\nu^{-a}\rho$ 
in  cuspidal support
in \eqref{formula-jac-pola-b-a-c},
left of $\otimes$,  
we have options
\begin{itemize}
    \item[$\bullet$]
    $r-b=-a$, so $b+1-s>\frac{1}{2}$, and we have $i=-\frac{1}{2}$, $s=b-a$ 
    and 
    $\sigma'=\delta([\nu^\frac{1}{2}\rho,\nu^c \rho]_+;\sigma)=\sigma_c$.
    \item[$\bullet$]
    $b+1-s=-a$, so $s>b+\frac{1}{2}$ and this is not possible.
    \item[$\bullet$]
    $-i=-a$, this is not possible since
    $\delta([\nu^{a+1}\rho,\nu^c\rho]_+;\sigma
        )$ is not defined.
\end{itemize}
Looking for $\ds{\frac{1}{2}}{a}\otimes \sigma_{b,c}^-$, we have
$r=b+\frac{1}{2}$ and $s=0$. 
Thus we search in
\[
\delta'\otimes \ds{\frac{1}{2}}{b}\rtimes \sigma'.
\]
Now in \eqref{jacquet-segment} we have $j=c$ and 
$i=-\frac{1}{2}$, so 
$\sigma'=\delta([\nu^\frac{1}{2}\rho,\nu^c \rho]_+;\sigma)=\sigma_c$.
But, $\sigma_{b,c}^-\nleq   \ds{\frac{1}{2}}{b} \rtimes\sigma_c$, and $\sigma_{b,c}^+$ appears there once, by
Lemma
\ref{lema-diskretna-podreprezentacija}.
\end{proof}

\begin{lemma} 
\label{jacq-diskretne-u-negativnoj}
We have in $R(G)$, with maximum multiplicities:
\begin{align*}
\mu^*
(\delta([\nu^\frac{1}{2}\rho,\nu^b\rho])\rtimes
\sigma_{a,c}^-)
\geq
 0 \cdot \ds{-a}{b}\otimes \sigma_c
+
0 \cdot
\ds{\frac{1}{2}}{a}\otimes \sigma_{b,c}^-.
\end{align*}
\end{lemma}
\begin{proof}
The proof goes as in Lemma 
\ref{jacquet-diskr-A1+},
with the difference that
one now obtains 
$\sigma'=\delta([\nu^\frac{1}{2},\rho,\nu^c]_-;\sigma)$, but this is not defined, by 
\eqref{segment-jacquet-equation}.
\end{proof}

Now we can determine all discrete series that appear in 
\eqref{A-A1}.

\begin{proposition} 
\label{diskretne-u-velikoj-prpopzicija}
Writting all discrete series, with multiplicities, we have in $R(G)$:
\begin{align*}
\delta([\nu^{-a}\rho,\nu^c\rho])\times \delta([\nu^\frac{1}{2}\rho,\nu^b\rho])\rtimes \sigma 
&\geq
\sigma^+_{b,c,a}+\sigma^-_{b,c,a},
\\
\ds{\frac{1}{2}}{b}
\rtimes \sigma_{a,c}^+
&\geq
\sigma^+_{b,c,a},
\\
\delta([\nu^\frac{1}{2}\rho,\nu^b\rho])\rtimes
      \sigma_{a,c}^-
 &\geq 0.     
\end{align*}
\end{proposition}

\begin{proof}
By \eqref{A-A1} and Lemma \ref{jacquet-discrete-A},
only possible discrete series subquotients in all of these representations are $\sigma^+_{b,c,a}$, 
$\sigma^-_{b,c,a}$ and $\sigma^-_{a,b,c}$.
By Theorem
 \ref{ciganovic-A2} and
\eqref{long-intertwining},
$\sigma^+_{b,c,a}$ and  
$\sigma^-_{b,c,a}$ appear in the first equation. 
Lemmas \ref{lema-pola-a-bc-tvrdnja} and  
\ref{dis-lema-jac-prva} show that they appear with multiplicity one.
Now \eqref{jacq-diskretne} and Lemma 
\ref{jacquet-diskr-A1+} 
show that we have only one discrete series in the second equation, $\sigma^+_{b,c,a}$,
and none in the last, by Lemma
\ref{jacq-diskretne-u-negativnoj}.
Since we have no $\sigma_{a,b,c}^-$ in the last two equations, 
Theorems
\ref{muic-A3} and \ref{ciganovic-A2}, and
\eqref{long-intertwining},
show that we have no 
$\sigma_{a,b,c}^-$ in the first equation. 
\end{proof}


\section{Non-tempered candidates}
\label{netemperirani}
As noted in Section \ref{dekompozicija1}, we  search for possible remaining non-tempered subquotients
in 
\[
\delta([\nu^\frac{1}{2}\rho,\nu^b\rho])\rtimes
      \sigma_{a,c}^+
      +
      \delta([\nu^\frac{1}{2}\rho,\nu^b\rho])\rtimes
      \sigma_{a,c}^-.
\]
We need a lemma.
\begin{lemma}   \label{cuvanje_predznaka} We have
\[
\sigma_{b,c}^- \not \leq 
\ds{a+1}{b}\rtimes \sigma_{a,c}^+
\quad \textrm{ and } \quad
\sigma_{b,c}^+ \not \leq 
\ds{a+1}{b}\rtimes \sigma_{a,c}^-.
\]
\end{lemma}

\begin{proof}
We prove the first claim, and the second follows similarly.  
Picking $i=-a-1$ and $j=c$, in \eqref{jacquet-segment}, we have
\[
\mu^*(\sigma_{b,c}^-)\geq
\ds{a+1}{b}\otimes \sigma_{a,c}^-.
\]
Now we show 
$
\ds{a+1}{b}\otimes \sigma_{a,c}^- \not \leq 
\mu^*(\ds{a+1}{b}\rtimes \sigma_{a,c}^+).
$
By \eqref{komnozenje} 
\begin{align*}
\mu^*(&\ds{x}{y}\rtimes \sigma)=
\sum_{\delta'\otimes\sigma'\leq \mu^*(\sigma_{a,c}^+)}
\sum_{r=0}^{b-a} \sum_{s=0}^{r}                                \\
&\ds{r-b}{-a-1}
\times
\ds{b+1-s}{b}
\times \delta'
\otimes
\ds{b+1-r}{b-s}\rtimes \sigma'.
\end{align*}
But 
$\ds{a+1}{b}\otimes \sigma_{a,c}^-$ can not appear here because
\begin{itemize}
    \item[a)]
 if $\delta'$ is not trivial, looking at \eqref{jacquet-segment}, it must contain 
in its cuspidal support at least one of the following: 
$\nu^a\rho$ or  $\nu^c\rho$,
and $\ds{a+1}{b}$ doesn't.
\item[b)] if $\delta'$ is trivial, then
$\sigma'=\sigma_{a,c}^+$, and looking at cuspidal support right of $\otimes$, we have $r=s$, but
${\ds{a+1}{b}\otimes \sigma_{a,c}^- \neq ...\otimes \sigma_{a,c}^+}$.
\end{itemize}
\end{proof}
Now we have
\begin{proposition} 
\label{nediskretni-kandidati-u-pozitivnoj}
If $\pi$ is a non-tempered subquotient of  
$\ds{\frac{1}{2}}{b}\rtimes
      \sigma_{a,c}^+$, different from its Langlands quotient, then $\pi$ is either
      $L(\ds{\frac{1}{2}}{a}\rtimes \sigma_{b,c}^+)$ or
      \\
      $L(\ds{-a}{b}\rtimes \sigma_c)$.
\end{proposition}

\begin{proof}
We use Lemma 2.2 of \cite{muic:composition series}
(in terms of that lemma 
$\pi \leq \ds{-l_1}{l_2}\rtimes \sigma$, $-l_1=\frac{1}{2}$, $l_2=b$ and $\sigma=\sigma_{a,c}^+$).
So we look for
possible embeddings
\begin{equation} \label{moguca_ulaganja}
\pi \hookrightarrow \ds{-\alpha_1}{\beta_1}\rtimes \pi',
\end{equation}
where $-\alpha_1+\beta_1<0$ and $\pi'$ is irreducible.
By the lemma, there exists an irreducible representation $\sigma_1$ such that
\begin{equation} \label{moguca_ulaganja_uvjet_mu}
    \begin{cases}
    \mu^*(\sigma_{a,c}^+)\geq \ds{\frac{1}{2}}{\beta_1}\otimes \sigma_1
    \\
    \pi'\leq \ds{\alpha_1+1}{b}\rtimes \sigma_1
    \end{cases}
\end{equation}
and we must have
\begin{equation} \label{moguca_ulaganja_uvjet_nejednakost}
\begin{cases}
-\frac{1}{2} \leq \beta_1  \\
b \geq  \alpha_1  > \beta_1, -\frac{1}{2}  \\
\alpha_1 \geq \frac{1}{2}.
\end{cases}
\end{equation}
We have cases:

\begin{itemize}
    \item[a)]
    $\beta_1=-\frac{1}{2}$. Now
    $\sigma_1=\sigma_{a,c}^+$.
    \begin{itemize}
    \item[$\bullet$] As in the lemma, we may assume that if $\pi'$ is tempered, then $2\alpha_1+1\in \textrm{Jord}_\rho(\sigma_{a,c}^+)$. 
    So, assume that $\pi'$ is tempered.
    Now $\alpha_1=c$ is not possible, since that would imply $b\geq c$. 
    Thus $\alpha_1=a$, and
    \[
    \pi'\leq \ds{a+1}{b}\rtimes \sigma_{a,c}^+.
    \]
    Looking at the cuspidal support on the right hand side,  Lemma 8.1 of \cite{tadic-diskretne} implies that $\pi'$ is a discrete series. Thus $\pi'=\sigma_{b,c}^+$ or $\pi'=\sigma_{b,c}^-$.
    By Lemma \ref{cuvanje_predznaka} we have $\pi'=\sigma_{b,c}^+$.
    So \eqref{moguca_ulaganja} is written as
    \[
    \pi \hookrightarrow \ds{-a}{-\frac{1}{2}}\rtimes \sigma_{b,c}^+.
    \]
    Thus, we have $\pi \cong 
    \textrm{L}(\ds{\frac{1}{2}}{a}\rtimes \sigma_{b,c}^+)$.
    \item[$\bullet$] If $\pi'$ is not tempered, by
    Lemma 2.2 of \cite{muic:composition series}, there exist
    $\beta_2+1$ and $(\beta_2+1)_-=(\beta_2)_-+1\in \textrm{Jord}_\rho(\sigma_{a,c}^+)=\{2a+1,2c+1 \}$,
    such that 
    \newline
    (in terms of the lemma $\alpha_1\leq (\beta_2)_-<\beta_2<\alpha_2\leq l_2$) we have
    \[
    \alpha_1 \leq a < c <\alpha_2 \leq b.
    \]
    So $c<b$, but this is a contradiction.
    \end{itemize}
    
    \item[b)] $\beta_1 >\frac{1}{2}$. Then, by the lemma, 
    $2\beta_1+1\in \textrm{Jord}_\rho(\sigma_{a,c}^+)=\{2a+1,2c+1 \}$.
    Since 
    \[
    c >b \geq \alpha_1 >\beta_1,-\frac{1}{2}
    \]
    we have $\beta_1=a$ and $\alpha_1> a$.
    Now \eqref{moguca_ulaganja_uvjet_mu} gives
    \[
      \mu^*(\sigma_{a,c}^+)\geq \ds{\frac{1}{2}}{a}\otimes \sigma_1.
    \]
    To determine $\sigma_1$, we look at $\mu^*(\sigma_{a,c}^+)$. In \eqref{jacquet-segment} 
    it is necessary to pick $j=c$ and $i=-\frac{1}{2}$, and we have 
    \[
    \sigma_1=
    \delta([\nu^\frac{1}{2}\rho,\nu^c\rho]_+;\sigma)
    =\sigma_c.
    \]
    So far, we have
    \begin{equation} 
    \begin{cases}
    \pi\hookrightarrow \ds{-\alpha_1}{a}\rtimes\pi',
    \\
    \pi'\leq \ds{\alpha_1+1}{b}\rtimes\sigma_c,
    \\
    a<\alpha_1 \leq b.
    \end{cases}
    \end{equation}
    If $\alpha_1=b$ then $\pi'\cong \sigma_c$ and 
    $\pi\hookrightarrow \ds{-b}{a}\rtimes\sigma_c
    $, so 
    \[
    \pi 
    \cong 
    \textrm{L}(\ds{-a}{b}\rtimes\sigma_c),
    \]
    as expected.
    Thus, we assume
    \[ 
    a<\alpha_1 < b.
    \]
    Since 
    $
    \textrm{Jord}_\rho(\sigma_c)
    \cap
    [2\alpha_1+1,2b+1]
    =\{2c+1\}\cap
    [2\alpha_1+1,2b+1]=
    \emptyset
    $,
    Proposition 3.1 ii) of \cite{muic:composition series}
    implies that $ \ds{\alpha_1+1}{b}\rtimes\sigma_c$
    is irreducible. So
    \[
    \pi'\cong 
    \ds{\alpha_1+1}{b}\rtimes\sigma_c
    \cong 
    \ds{-b}{-\alpha_1-1}\rtimes\sigma_c,
    \]
    and finally
    \[
    \pi\hookrightarrow \ds{-\alpha_1}{a}
    \times
    \ds{-b}{-\alpha_1-1}\rtimes\sigma_c.
    \]
    We want to prove that the representation on the right has a unique irreducible subrepresentation, and that it is 
    $L(\ds{-a}{b}\rtimes\sigma_c)$.
    Consequently $\pi\cong L(\ds{-a}{b}\rtimes\sigma_c)$ and the proof of the proposition will be over.
    First, we see that
    \begin{equation*}
    \begin{split}
    L(\ds{-a}{b}\rtimes\sigma_c)
    &\hookrightarrow
    \ds{-b}{a}\rtimes\sigma_c
    \\
    &\hookrightarrow
    \ds{-\alpha_1}{a}
    \times
    \ds{-b}{-\alpha_1-1}\rtimes\sigma_c.
    \end{split}
    \end{equation*}
    Now, to prove 
    that $\ds{-\alpha_1}{a}
    \times
    \ds{-b}{-\alpha_1-1}\rtimes\sigma_c$ has a unique irreducible subrepresentation,
    it is enough to 
    see that 
    \[
    \ds{-\alpha_1}{a}
    \otimes
    \ds{-b}{-\alpha_1-1}\rtimes\sigma_c
    \]
    appears once in
    \(
    \mu^*(\ds{-\alpha_1}{a}
    \times
    \ds{-b}{-\alpha_1-1}\rtimes\sigma_c)
    \).
\end{itemize}
By \eqref{komnozenje}, we look for
\(
0\leq s\leq r \leq a+\alpha_1+1,
0\leq v\leq u \leq b-\alpha_1
\) 
and
\(
\delta'\otimes\sigma' \leq \mu^*(\sigma_c)\) such that
\begin{equation} \label{konacno-1}
\begin{split}
    \ds{-\alpha_1}{a}
    \leq
    &\ds{r-a}{\alpha_1}\times\ds{a+1-s}{a}
    \\
    \times 
    &\ds{u+\alpha_1+1}{b}
    \times\ds{-\alpha_1-v}{-\alpha_1-1}
    \times \delta' \quad \textrm{ and }
\end{split}    
\end{equation}
\begin{equation} \label{konacno-2}
\begin{split}
    \ds{-b}{-\alpha_1-1}
    \rtimes\sigma_c
    \leq &
    \\
    \ds{a+1-r}{a-s}
    & \times
    \ds{-\alpha_1-u}{\alpha_1-1-v}
     \rtimes \sigma'.
\end{split}    
\end{equation}
We compare cuspidal support.
On the right hand side of \eqref{konacno-2} only 
$\sigma'$ may have $\nu^{\pm c}\rho$ in its cuspidal support. Since it appears on the left hand side, looking at \eqref{jacquet-strogo-pozitivna}
we see that we must have $\sigma'=\sigma_c$, and $\delta'=1$.

Now we look at \eqref{konacno-1}. We can not have
$\nu^b\rho$ in the cuspidal support on the right hand side, since it doesn't exist on the left hand side. So $u=b-\alpha_1$. Similarly, on the right hand side, we can not have $\nu^{-\alpha_1-1}\rho$ in the cuspidal support, so $v=0$.

Further, we look for $\nu^{-\alpha_1}\rho$ on the right hand side of \eqref{konacno-1}. We must have either $-\alpha_1=r-a$ or $-\alpha_1=a+1-s$. 
But $-\alpha_1=r-a$ implies $r=a-\alpha_1<0$, which is not possible. So we have $-\alpha_1=a+1-s$, that is $s=a+1+\alpha_1$, which forces $r=a+1+\alpha_1$.
The uniqueness is proved. 
\end{proof}
Similarly, we have
\begin{proposition} 
\label{nediskretni-kandidati-u-negativnoj}
If $\pi$ is a non-tempered subquotient of  
$\delta([\nu^\frac{1}{2}\rho,\nu^b\rho])\rtimes
      \sigma_{a,c}^-$, different from its Langlands quotient, then $\pi$ is
      $L(\ds{\frac{1}{2}}{a}\rtimes \sigma_{b,c}^-)$.
\end{proposition}
\begin{proof}
We follow the same lines as in the proof of Proposition 
\ref{nediskretni-kandidati-u-pozitivnoj}, with the difference that 
we have $\sigma_{a,c}^-$ instead of $\sigma_{a,c}^+$. Now, part a)  of that proof gives our candidate, while b) part, at its very beginning, implies
$\sigma_1=
    \delta([\nu^\frac{1}{2}\rho,\nu^c\rho]_-;\sigma)$
    which is not possible, see \eqref{segment-jacquet-equation}.
\end{proof}


\section{Multiplicity of 
$L(\ds{\frac{1}{2}}{a}
    \rtimes \sigma_{b,c}^\pm)$}
\label{mult_L1}
Here we write explicitly 
$L(\ds{\frac{1}{2}}{a}
    \rtimes \sigma_{b,c}^\pm
    )$
as a
non-tempered subquotient of
$\ds{\frac{1}{2}}{b}\rtimes \sigma_{a,c}^\pm$,
 different from its unique Langlands quotient, as claimed by Lemma 6.2 of \cite{muic:reducibility principal}.
We start with a couple of lemmas.

\begin{lemma} \label{male-diskretne-multiplicitet}
  Discrete series $\sigma_{b,c}^+$ and $\sigma_{b,c}^-$ appear with multiplicity one in equations
  \begin{equation*}
     \begin{split}
     \sigma_{b,c}^+ \leq \ds{a+1}{b}\rtimes 
\sigma_{a,c}^+,
     \\
     \sigma_{b,c}^- \leq \ds{a+1}{b}\rtimes 
\sigma_{a,c}^-,
      \\
      \sigma_{b,c}^+ +\sigma_{b,c}^- \leq \
\ds{a+1}{b}
\times \ds{-a}{c}\rtimes \sigma.
    \end{split}
  \end{equation*}  
\end{lemma}
\begin{proof}
    We start with the last equation. First, observe that in $R(G)$
    \begin{equation*}
\begin{split}
\sigma_{b,c}^+
+
\sigma_{b,c}^-
\leq 
\ds{-b}{c}\rtimes \sigma
\leq& 
\ds{-a}{c}
\times \ds{-b}{-a-1}\rtimes \sigma
\\
=&
\ds{a+1}{b}
\times \ds{-a}{c}\rtimes \sigma.
\end{split}
\end{equation*}
Further, replace $a$ with $b$ in \eqref{jacquet-segment} and take there $i=-a-1$, to obtain
\[
\ds{a+1}{b}\times \ds{-a}{c}\otimes \sigma
\leq \mu^*(\sigma_{b,c}^\pm).
\]
Now, it is enough to see that multiplicity of this summand is two in
\begin{equation*}
\begin{split}
\mu^*(\ds{a+1}{b}
\times \ds{-a}{c}\rtimes \sigma)
\geq & \sum_{i=0}^{b-a} \quad  \sum_{j=0}^{c+a+1} 
\\
\ds{i-b}{-a-1} \times
    \ds{b+1-i}{b} \times &
    \ds{-c+j}{a} \times 
    \ds{c+1-j}{c} \otimes \sigma,
\end{split}
\end{equation*}
writting only summands of type $...\otimes \sigma$.
We can't have $\nu^{-a-1}\rho$ in the cuspidal support left of $\otimes$, so $i=b-a$. Now $j=c-a$ or $a+c+1$. 

  To prove first two equations, observe that 
  $\mu^*(\sigma_{a,c}^\pm)\geq \ds{-a}{c}\otimes \sigma$ implies   
\[
     \ds{a+1}{b}\times \ds{-a}{c}\otimes \sigma
     \leq\mu^*(\ds{a+1}{b}\rtimes 
    \sigma_{a,c}^\pm).
\]
Since 
$\ds{a+1}{b}\rtimes \sigma_{a,c}^\pm
\leq 
\ds{a+1}{b} \times \ds{-a}{c}\rtimes \sigma$, the first part of the proof and Lemma \ref{cuvanje_predznaka} complete our proof.
\end{proof}

\begin{lemma} \label{lema-nediskretni-jacq1}
Both
$\ds{-a}{-\frac{1}{2}}\otimes \sigma_{b,c}^+$ 
and
$\ds{-a}{-\frac{1}{2}}\otimes \sigma_{b,c}^-$ 
appear with multiplicity one in
$\mu^*(\delta([\nu^{-a}\rho,\nu^c\rho])\times \delta([\nu^\frac{1}{2}\rho,\nu^b\rho])\rtimes \sigma)$.
\end{lemma}

\begin{proof}
By \eqref{komnozenje}, we look for
$0\leq s \leq r \leq b+\frac{1}{2}$ and 
$0\leq v \leq u \leq a+c+1$ such that 
\begin{equation} \label{NEDISKR-lema-jac-prva-jed1}
\begin{split}
    \ds{-a}{-\frac{1}{2}}
    \leq 
    &\ds{r-b}{-\frac{1}{2}} \times
    \ds{b+1-s}{b} \times
    \\
    &\ds{-c+u}{a} \times
    \ds{c+1-v}{c} \quad \textrm{and}
\end{split}    
\end{equation}
\begin{equation} \label{NEDISKR-lema-jac-prva-jed2}
\sigma_{b,c}^{\pm} \leq 
\ds{b+1-r}{b-s}
\times \ds{c+1-u}{c-v}\rtimes \sigma,
\end{equation}
with sign in  \eqref{NEDISKR-lema-jac-prva-jed2}
depending on the subquotient we are looking for.
Comparing cuspidal support in 
\eqref{NEDISKR-lema-jac-prva-jed1}
gives $v=s=0$ and $u=c+a+1$. Now $r=b-a$ and 
\eqref{NEDISKR-lema-jac-prva-jed2} is
\begin{equation} \label{NEDISKR-lema-jac-prva-jed2-reduc}
\sigma_{b,c}^{\pm} \leq 
\ds{a+1}{b}
\times \ds{-a}{c}\rtimes \sigma.
\end{equation}
By  Lemma \ref{male-diskretne-multiplicitet},
 both $\sigma_{b,c}^+$ and $\sigma_{b,c}^-$ occur in 
\eqref{NEDISKR-lema-jac-prva-jed2-reduc} once.
\end{proof}

\begin{lemma} \label{lema-nediskretni-jacq2}
For $\epsilon=\pm $ irreducible representation
$\ds{-a}{-\frac{1}{2}}\otimes \sigma_{b,c}^\epsilon$ 
appears with multiplicity one in
$\mu^*(\ds{\frac{1}{2}}{b}\rtimes 
\sigma_{a,c}^\epsilon)$.
\end{lemma}
\begin{proof}
For simplicity, assume $\epsilon=+$, a proof for $\epsilon=-$ is just an appropriate sign change.
By Lemma \ref{lema-nediskretni-jacq1}, it is enough to prove the existence. Use formula \eqref{formula-jac-pola-b-a-c}, and pick there $s=0$, $r=b-a$ and $\delta\otimes\sigma'=1\otimes \sigma_{a,c}^+$. So we need to prove
\[
\ds{-a}{-\frac{1}{2}}\otimes \sigma_{b,c}^+
\leq 
\ds{-a}{-\frac{1}{2}}\otimes
\ds{a+1}{b}\rtimes 
\sigma_{a,c}^+,
\]
with multiplicity one. Equivalently 
\[
\sigma_{b,c}^+\leq \ds{a+1}{b}\rtimes 
\sigma_{a,c}^+,
\]
and now Lemma \ref{male-diskretne-multiplicitet} completes our proof.
\end{proof}

Now we have

\begin{proposition} \label{prop-l-pola-a-b-c}
With all multiplicities being one, we have in $R(G)$ 
\begin{align*}
    \ds{\frac{1}{2}}{b}
    \rtimes \sigma_{a,c}^+
    \geq & 
    L(
    \ds{\frac{1}{2}}{a}
    \rtimes \sigma_{b,c}^+
    ),
\\
    \ds{\frac{1}{2}}{b}
    \rtimes \sigma_{a,c}^-
    \geq & 
    L(
    \ds{\frac{1}{2}}{a}
    \rtimes \sigma_{b,c}^-
    ),
\\
    \delta([\nu^{-a}\rho,\nu^c\rho])\times \delta([\nu^\frac{1}{2}\rho,\nu^b\rho])\rtimes \sigma
    \geq &
        L(
        \ds{\frac{1}{2}}{a}
        \rtimes \sigma_{b,c}^+
        )
    +
\\
        &
        L(
        \ds{\frac{1}{2}}{a}
        \rtimes \sigma_{b,c}^-
        ).
\end{align*} 
\end{proposition}
\begin{proof} First, observe that, by 
\eqref{long-intertwining} and Theorem
\ref{ciganovic-A2},
$\textrm{L}(
    \ds{\frac{1}{2}}{a}
    \rtimes \sigma_{b,c}^\epsilon
    )$
    is contained in 
    $\delta([\nu^{-a}\rho,\nu^c\rho])\times \delta([\nu^\frac{1}{2}\rho,\nu^b\rho])\rtimes \sigma$.
    Existence and
    multiplicity one now follow from
    Lemmas
    \ref{lema-nediskretni-jacq1}
    and
    \ref{lema-nediskretni-jacq2}.
\end{proof}



\section{Multiplicity of 
$L(\ds{-a}{b}\rtimes \sigma_c)$ }
\label{mult_L2}

By  Theorem \ref{muic-A3} and 
 \eqref{long-intertwining},
this subquotient does appear in 
$\ds{-a}{c}\times \ds{\frac{1}{2}}{b}
\rtimes \sigma$, but also as a candidate in 
$\ds{\frac{1}{2}}{b}\rtimes
      \sigma_{a,c}^+$, by Proposition
\ref{nediskretni-kandidati-u-pozitivnoj},
so we need to check its multiplicity.
First we prove a couple of lemmas to obtain a subquotient in some of its Jacquet module, used to identify it.

\begin{lemma} \label{lema-jedan-pola-unitarna}
    In appropriate Grothendieck group we have 
    \begin{align*}
    \nu^\frac{1}{2}\rho\times 
    \nu^\frac{1}{2}\rho 
    \rtimes \sigma=
        \delta([
        \nu^{-\frac{1}{2}}\rho,
        \nu^\frac{1}{2} \rho]_+;\sigma)
    +
        \delta([
        \nu^{-\frac{1}{2}}\rho,
        \nu^\frac{1}{2} \rho]_-;\sigma)
    +
    \\
        L(\nu^\frac{1}{2}\rtimes 
                \sigma_{\frac{1}{2}})
    +
        L(
        \nu^\frac{1}{2}\rho\times 
        \nu^\frac{1}{2}\rho 
        \rtimes \sigma
        ).
    \end{align*}
\end{lemma}
\begin{proof}
    We need to prove that there are no other irreducible subquotients, as stated are obvious.
    In appropriate Grothendieck group we have
    \begin{align*}
    \nu^\frac{1}{2}\rho\times 
    \nu^\frac{1}{2}\rho 
    \rtimes \sigma
    =
    \nu^\frac{1}{2}\rho\times 
    \nu^{-\frac{1}{2}}\rho 
    \rtimes \sigma
    =
    \\
    \ds{-\frac{1}{2}}{\frac{1}{2}}\rtimes \sigma
    +
    L(\nu^{-\frac{1}{2}}\rho,\nu^\frac{1}{2}\rho)
    \rtimes \sigma.
    \end{align*}
    Here 
    \[
    \ds{-\frac{1}{2}}{\frac{1}{2}}\rtimes \sigma=
     \delta([
        \nu^{-\frac{1}{2}}\rho,
        \nu^\frac{1}{2} \rho]_+;\sigma)
    +
        \delta([
        \nu^{-\frac{1}{2}}\rho,
        \nu^\frac{1}{2} \rho]_-;\sigma).
    \]
    On the other hand,
    the representation
    \[
    L(\nu^{-\frac{1}{2}}\rho,\nu^\frac{1}{2}\rho)
    \hookrightarrow
    \nu^{-\frac{1}{2}}\rho \times \nu^\frac{1}{2}\rho
    \] 
    is unitary (see \cite{tadic:unitary}, Introduction). 
    So is 
    \[
    L(\nu^{-\frac{1}{2}}\rho,\nu^\frac{1}{2}\rho)\rtimes  \sigma,
    \]
    and its every irreducible subquotient contains
    $L(\nu^{-\frac{1}{2}}\rho,\nu^\frac{1}{2}\rho)
    \otimes  \sigma$ in an appropriate Jacquet module.
    Now, one can check that 
    $L(\nu^{-\frac{1}{2}}\rho,\nu^\frac{1}{2}\rho)
    \otimes  \sigma$
    appears with multiplicity two in 
    $\mu^*(\nu^\frac{1}{2}\rho\times 
        \nu^\frac{1}{2}\rho 
        \rtimes \sigma)$, but does not appear in 
    $\mu^*(\ds{-\frac{1}{2}}{\frac{1}{2}}
    \rtimes \sigma)$.

\end{proof}

\begin{lemma} \label{lema-jedan-pola-do-b-unitarna}
    We have in $R(G)$, with multiplicity one 
    \[
    \ds{\frac{1}{2}}{b}\times \ds{\frac{1}{2}}{b} \rtimes \sigma
    \geq 
    L(\ds{\frac{1}{2}}{b} \rtimes \sigma_b).
    \]
\end{lemma}
\begin{proof} 
By Lemma \ref{lema-jedan-pola-unitarna} we may assume 
$b\geq \frac{3}{2}$. 
Obviously, 
$L(\ds{\frac{1}{2}}{b} \rtimes \sigma_b)$ 
does appear as   
a subquotient, so we need to prove multiplicity one. 
Let us denote 
\[
    \pi=    
    \ds{\frac{3}{2}}{b}
    \times
    \ds{\frac{3}{2}}{b}
    \otimes
    L(\nu^\frac{1}{2}\rho\rtimes \sigma_{\frac{1}{2}}
    ).
\]
It is enough to see that $\pi$ appears in both 
\[
\textrm{i)} \textrm{\ }
 \mu^*(L(\ds{\frac{1}{2}}{b} \rtimes \sigma_b))
\quad
\textrm{ and }
\quad
\textrm{ii)}\textrm{\ } 
\mu^*(\ds{\frac{1}{2}}{b}\times \ds{\frac{1}{2}}{b} 
        \rtimes \sigma),
\]
and that its multiplicity is one in ii).
\\
We start with i).
By Theorem 5.1 of \cite{muic:composition series} ii), and 
Lemma 5.4 there, we have in $R(G)$
\begin{equation*} 
    \ds{\frac{1}{2}}{b}\rtimes
      \sigma_{b}
    =
    \sigma_{temp}
    +
    L(\ds{\frac{1}{2}}{b}\rtimes
      \sigma_{b}),
\end{equation*}
where $\sigma_{temp}$ is an irreducible tempered 
subquotient of $\ds{-b}{b}\rtimes \sigma$.
As 
\[
    ...\otimes 
    L(\nu^\frac{1}{2}\rho\rtimes \sigma_{\frac{1}{2}})
\overset{\eqref{segment-jacquet-equation}}{\nleq}
    ...\otimes
    \delta([\nu^{b+1-i}\rho,\nu^{b-j}\rho])\rtimes \sigma
 \overset{\eqref{komnozenje}}{\leq}
    \mu^*(\ds{-b}{b}\rtimes \sigma),
\]
we see that 
$\pi$ does not appear in
    $\mu^*(\ds{-b}{b}\rtimes \sigma)$.
On the other hand, 
by
\eqref{jacquet-strogo-pozitivna}
      ${\mu^*(\sigma_b)\geq 
      \ds{\frac{3}{2}}{b}\otimes\sigma_\frac{1}{2}}$, so taking
$i=b+\frac{1}{2}$, $j=b-\frac{1}{2}$ in \eqref{komnozenje}
we have 
$\mu^*(\ds{\frac{1}{2}}{b}\rtimes
      \sigma_{b})
\geq 
    \ds{\frac{3}{2}}{b}
    \times
    \ds{\frac{3}{2}}{b}
    \otimes
    \nu^\frac{1}{2}\rho\rtimes \sigma_{\frac{1}{2}}
\geq
    \pi$.

ii) we look for 
$0\leq i\leq j \leq  b+\frac{1}{2}$ and $0\leq s\leq r \leq b+ \frac{1}{2}$
such that
\begin{align*}
\ds{\frac{3}{2}}{b}
    \times
    \ds{\frac{3}{2}}{b}
\leq 
        &\ds{i-b}{-\frac{1}{2}}
    \times
        \ds{b+1-j}{b}
    \\
    \times
        &\ds{r-b}{-\frac{1}{2}}
    \times
        \ds{b+1-s}{b},        
\end{align*}
and
\begin{align*}
L(\nu^\frac{1}{2}\rho\rtimes \sigma_{\frac{1}{2}})
\leq
\ds{b+1-i}{b-j}
\times
\ds{b+1-r}{b-s}
    \rtimes
        \sigma.
\end{align*}
The first equation implies $i=r=b+\frac{1}{2}$ and 
$j=s=b-\frac{1}{2}$. 
The second is
\[
L(\nu^\frac{1}{2}\rho\rtimes \sigma_{\frac{1}{2}})
\leq 
\nu^\frac{1}{2}\rho\times 
    \nu^\frac{1}{2}\rho 
    \rtimes \sigma.
\]
Multiplicity one in the second follows from Lemma
\ref{lema-jedan-pola-unitarna}.
\end{proof}

\begin{lemma} \label{lema-mult-L-dva-seg}
    We have in $R(G)$, with multiplicity one
    \[
    \ds{\frac{1}{2}}{c}
    \times
    \ds{\frac{1}{2}}{b}
    \rtimes \sigma
    \geq
    L(\ds{\frac{1}{2}}{b} \rtimes \sigma_c).
    \]
\end{lemma}
\begin{proof}
Obviously, 
$L(\ds{\frac{1}{2}}{b} \rtimes \sigma_c)$ 
does appear as   
a subquotient, so we need to prove multiplicity one. 
Let us denote 
\[
    \xi=    
    \ds{b+1}{c}
    \otimes
    L(\ds{\frac{1}{2}}{b} \rtimes \sigma_b).
\]
It is enough to see that $\xi$ appears in both 
\[
\textrm{i)} \textrm{\ }
 \mu^*(L(\ds{\frac{1}{2}}{b} \rtimes \sigma_c))
\quad
\textrm{ and }
\quad
\textrm{ii)}\textrm{\ } 
\mu^*(\ds{\frac{1}{2}}{c}\times \ds{\frac{1}{2}}{b} 
        \rtimes \sigma),
\]
and that its multiplicity is one in ii).
\\
We start with i).
By Lemma \ref{lema-diskretna-podreprezentacija}, we have in $R(G)$
\begin{equation*} 
    \ds{\frac{1}{2}}{b}\rtimes
      \sigma_{c}
    =
    \sigma_{b,c}^+
    +
    L(\ds{\frac{1}{2}}{b}\rtimes
      \sigma_{c}).
\end{equation*}
where $\sigma_{b,c}^+\leq \ds{-b}{c}\rtimes \sigma$.
As 
\[
    ...\otimes 
        L(\ds{\frac{1}{2}}{b} \rtimes \sigma_b)
\overset{\eqref{segment-jacquet-equation}}{\nleq}
    ...\otimes
    \delta([\nu^{c+1-i}\rho,\nu^{c-j}\rho])\rtimes \sigma
 \overset{\eqref{komnozenje}}{\leq}
    \mu^*(\ds{-b}{c}\rtimes \sigma),
\]
we see that 
$\xi$ does not appear in
    $\mu^*(\ds{-b}{c}\rtimes \sigma)$.
On the other hand, 
by
\eqref{jacquet-strogo-pozitivna}
      $\mu^*(\sigma_c)
      \geq 
      \ds{b+1}{c}\otimes\sigma_b$,
      so taking
$i=b+\frac{1}{2}$, $j=0$ in \eqref{komnozenje}
we have 
$\mu^*(\ds{\frac{1}{2}}{b}\rtimes
      \sigma_{c})
    \geq
      \ds{b+1}{c}
        \otimes
      \ds{\frac{1}{2}}{b} \rtimes \sigma_b
    \geq   
      \xi$.

ii) 
we look for 
    $0\leq i\leq j \leq c+ \frac{1}{2}$ 
and 
    $0\leq s\leq r \leq b+ \frac{1}{2}$
such that 
\begin{align*}
\ds{b+1}{c}
\leq 
        &\ds{i-c}{-\frac{1}{2}}
    \times
        \ds{c+1-j}{c}
    \\
    \times
        &\ds{r-b}{-\frac{1}{2}}
    \times
        \ds{b+1-s}{b},        
\end{align*}
and
\begin{align*}
L(\ds{\frac{1}{2}}{b} \rtimes \sigma_b)
\leq
\ds{c+1-i}{c-j}
\times
\ds{b+1-r}{b-s}
    \rtimes
        \sigma.
\end{align*}
The first equation implies 
$i=c+\frac{1}{2}$,
$r=b+\frac{1}{2}$,
$j=c-b$, and 
$s=0$. The second is 
\[
L(\ds{\frac{1}{2}}{b} \rtimes \sigma_b)
\leq
\ds{\frac{1}{2}}{b}\times \ds{\frac{1}{2}}{b} \rtimes \sigma.
\]
Multiplicity one in the second follows from Lemma
\ref{lema-jedan-pola-do-b-unitarna}.

\end{proof}
Now we provide a subquotient that can be used to identify $L(\ds{-a}{b}\rtimes \sigma_c)$.
\begin{lemma} \label{lema-Labc-jac}
We have in $R(G)$
\[
\mu^*(
L(\ds{-a}{b}\rtimes \sigma_c)
)
\geq 
\ds{\frac{1}{2}}{a}\otimes
L(\ds{\frac{1}{2}}{b} \rtimes \sigma_c).
\]
\end{lemma}
\begin{proof}
Let us denote
$\pi=
\ds{\frac{1}{2}}{a}\otimes L(\ds{\frac{1}{2}}{b} \rtimes \sigma_c)$.
By \ref{druga}, we have
\[
\delta([\nu^{-a}\rho,\nu^b\rho])\rtimes  \sigma_c=
\sigma_{b,c,a}^+ +
\sigma_{a,b,c}^- +
L(\delta([\nu^{-a}\rho,\nu^b\rho])\rtimes \sigma_c).
\]
Taking $x=-a,y=b$, $i=b+\frac{1}{2}, j=0$, $\sigma=\sigma'=\sigma_c$ in \eqref{komnozenje}, we have 
\[
    \mu^*(\delta([\nu^{-a}\rho,\nu^b\rho])
\rtimes  
    \sigma_c)
\geq
    \pi.
\]
Now, it is enough to see that neither 
$\mu^*(\sigma_{b,c,a}^+)$,  nor 
$\mu^*(\sigma_{a,b,c}^-)$ contain $\pi$. 
Since $\sigma_c \hookrightarrow \ds{\frac{1}{2}}{c}\rtimes \sigma$ (Theorem 
\ref{muic-diskretne-podreprezentacije}),
taking the Jacquet module with respect 
to  the minimal parabolic subgroup, we have
\begin{align*}
    s.s.r_{(m_\rho,\ldots, m_\rho)}(\pi)
\geq
       \nu^a\rho
        \otimes \cdots \otimes 
        \nu^\frac{1}{2}\rho
        \otimes 
        \nu^{-\frac{1}{2}}\rho
        \otimes \cdots \otimes 
        \nu^{-b} \otimes 
        \nu^c\rho
        \otimes \cdots \otimes 
        \nu^\frac{1}{2}\rho
        \otimes \sigma.
\end{align*}
Recall that $\rho \in GL(m_\rho,F)$.The following  
\[
m_\rho(a+ \cdots +\frac{1}{2}
        +
        (-\frac{1}{2})
        + \cdots  +
        (-b)) <0,
\]
contradicts a square integrability criterium, see the end of  Section 2 of \cite{tadic:family}.
So 
\[
\mu^*(\sigma_{b,c,a}^+)\ngeq \pi,
\textrm{ and }
\mu^*(\sigma_{a,b,c}^-)\ngeq \pi.
\]

\end{proof}

\begin{lemma} 
\label{lema-Labc-jac-u-velikoj}
We have in $R(G)$, with maximum multiplicity
\[
\mu^*(\ds{-b}{c}\times \ds{\frac{1}{2}}{a}\rtimes \sigma)
\geq
\ds{\frac{1}{2}}{a}\otimes L(\ds{\frac{1}{2}}{b} \rtimes \sigma_c).
\]
\end{lemma}

\begin{proof}
By \eqref{komnozenje},  we look for
$0\leq s \leq r \leq b+\frac{1}{2}$ and 
$0\leq v \leq u \leq a+c+1$, such that 
\begin{equation*} 
\begin{split}
    \ds{\frac{1}{2}}{a}
    \leq 
    &\ds{r-b}{-\frac{1}{2}} \times
    \ds{b+1-s}{b} \times
    \\
    &\ds{-c+u}{a} \times
    \ds{c+1-v}{c},
\end{split}    
\end{equation*}
and
\begin{equation*} 
L(\ds{\frac{1}{2}}{b} \rtimes \sigma_c) \leq 
\ds{b+1-r}{b-s}
\times \ds{c+1-u}{c-v}\rtimes \sigma.
\end{equation*}
The first equation implies $r=b+\frac{1}{2}$ and $v=s=0$, so $u=c+\frac{1}{2}$. Now, multiplicity one in the second equation follows from Lemma \ref{lema-mult-L-dva-seg}.
\end{proof}

\begin{lemma} 
\label{lema-Labc-jac-u-manjoj}
We have in $R(G)$, with maximum multiplicity
\[
\mu^*(\ds{\frac{1}{2}}{b}\rtimes \sigma_{a,c}^+)
\geq
\ds{\frac{1}{2}}{a}\otimes L(\ds{\frac{1}{2}}{b} \rtimes \sigma_c).
\]
\end{lemma}
\begin{proof}
 By \eqref{jacq-diskretne}    
$\mu^*( \sigma_{a,c}^+) \geq \ds{\frac{1}{2}}{a}\otimes \sigma_c.
$
The inequality follows by \eqref{komnozenje} 
and multiplicity one by \eqref{A-A1} and  
Lemma \ref{lema-Labc-jac-u-velikoj}.
\end{proof}

Finally we have

\begin{proposition} \label{propozicija L ab, sigma c}
Both induced representations
\[
\ds{-b}{c}\times \ds{\frac{1}{2}}{a}\rtimes \sigma
\textrm{ and }
\ds{\frac{1}{2}}{b}\rtimes \sigma_{a,c}^+,
\]
contain $L(\ds{-a}{b}\rtimes \sigma_c)$ as a subquotient, 
with multiplicity one.
\end{proposition}
\begin{proof}
 Theorem \ref{muic-A3} and 
 \eqref{long-intertwining}
 give existence in the first representation, and  Lemmas 
\ref{lema-Labc-jac}
and
\ref{lema-Labc-jac-u-velikoj} multiplicity one.
For the second, use
\eqref{A-A1}, and Lemma
\ref{lema-Labc-jac-u-manjoj}
\end{proof}

\section{Multiplicity of $L(\ds{-b}{c}\rtimes \sigma_a)$}
\label{mult_L3}

We observe that by  Theorems 
\ref{muic-A3} and
\ref{ciganovic-A2},
$L(\ds{-b}{c}\rtimes \sigma_a)$ appears  two times in 
\eqref{long-intertwining}. So we determine its multiplicity.
\begin{lemma} 
\label{lema-jacq-Lbc-a}
We have in $R(G)$, with  multiplicity one
\[
\mu^*(
L(\ds{-b}{c}\rtimes \sigma_a)
)
\geq
\ds{\frac{1}{2}}{a}\otimes L(\ds{-b}{c}\rtimes \sigma).
\]    
\end{lemma}
\begin{proof}
    Denote the subquotient by $\pi$.
    By \eqref{jacquet-strogo-pozitivna}
    $\mu^*(\sigma_a)\geq \ds{\frac{1}{2}}{a}\otimes \sigma$, so 
\[
\mu^*(
\ds{-b}{c}\rtimes \sigma_a
)
\geq
\pi. 
\]
Comparing Proposition \ref{druga} and Lemma \ref{lema-pola-a-bc-tvrdnja}, it is enough to show
\[
\mu^*(\ds{\frac{1}{2}}{a} \rtimes \sigma_{b,c}^\pm) \ngeq \pi.
\]
So by \eqref{komnozenje} consider $0\leq j\leq i \leq a+\frac{1}{2}$,
    $\delta'\otimes \sigma' \leq \mu^*(\sigma_{b,c}^\pm)$ and
\begin{align*}
    \ds{i-a}{-\frac{1}{2}}
    \times
    \ds{a-j+1}{a}
     \times \delta'
    \otimes 
    \ds{a+1-i}{a-j}
     \rtimes\sigma'.
\end{align*}    
As $\delta'$ should not contain neither $\nu^b\rho$ nor $\nu^c\rho$ in its cuspidal support, by \eqref{jacquet-segment}, we have 
$\delta'\otimes \sigma' =1\otimes \sigma_{b,c}^\pm$, and $i=j=a+\frac{1}{2}$, giving  $\ds{\frac{1}{2}}{a} \otimes \sigma_{b,c}^\pm \ncong \pi $.
\end{proof}

\begin{proposition}
\label{prop-mult-L-bc-a}
We have in $R(G)$, with multiplicity one
    \[
                    \ds{-a}{c}
                \times
                    \ds{\frac{1}{2}}{b}
            \rtimes
                \sigma
        \geq
            L(\ds{-b}{c}\rtimes \sigma_a).
    \]
\end{proposition}

\begin{proof}
Apply Lemma \ref{lema-jacq-Lbc-a} on
\eqref{dis-lema-jac-prva-f1}.
We have $r=b+\frac{1}{2}$, $s=0$, $v=0$, 
$u=c+\frac{1}{2}$.
By \eqref{lemma-diskr-u-dva-seg-kusp-f1}
 $
    L(\ds{-b}{c}\rtimes \sigma)
\leq
        \ds{\frac{1}{2}}{b}
    \times
        \ds{\frac{1}{2}}{c}
    \rtimes 
        \sigma
$
appears once.
\end{proof}

\section{Composition factors}
\label{kompozicijski_faktori}
Here we determine composition factors of the kernel $K_1$  from Section \ref{dekompozicija1}.
\begin{proposition} 
\label{prop-kompozicijski-fakt1}
    We have in $R(G)$
    \begin{align*}
\ds{\frac{1}{2}}{b} \rtimes \sigma_{a,c}^+
=
&L(\ds{\frac{1}{2}}{b} \rtimes \sigma_{a,c}^+)+
\sigma^+_{b,c,a}+
\\
&L(\ds{\frac{1}{2}}{a}\rtimes \sigma_{b,c}^+)+
L(\ds{-a}{b}\rtimes \sigma_c).
    \end{align*}
\end{proposition}
\begin{proof}
    Discrete series subquotients are determined in 
    Proposition \ref{diskretne-u-velikoj-prpopzicija}.
    Remaining irreducible subquotients are described by 
    Propositions
    \ref{nediskretni-kandidati-u-pozitivnoj},
    \ref{prop-l-pola-a-b-c}
    and
    \ref{propozicija L ab, sigma c}.
\end{proof}

\begin{proposition} 
\label{prop-kompozicijski-fakt2}
    We have in $R(G)$
    \begin{align*}
\ds{\frac{1}{2}}{b} \rtimes \sigma_{a,c}^-
=
L(\ds{\frac{1}{2}}{b} \rtimes \sigma_{a,c}^-)+ 
L(\ds{\frac{1}{2}}{a}\rtimes \sigma_{b,c}^-).    
    \end{align*}
\end{proposition}
\begin{proof}
    Possible series subquotients are discussed in 
    Proposition \ref{diskretne-u-velikoj-prpopzicija}.
    Remaining irreducible subquotients are described by 
    Propositions 
    \ref{nediskretni-kandidati-u-negativnoj}
    and
    \ref{prop-l-pola-a-b-c}.
\end{proof}

\begin{theorem}
\label{glavni-rezultat}
We have in $R(G)$
\begin{align*}
 \ds{-a}{c}\times
&\ds{\frac{1}{2}}{b}\rtimes \sigma =
L(\ds{-a}{c}\times
\ds{\frac{1}{2}}{b}\rtimes \sigma)+
\\
&L(\ds{\frac{1}{2}}{b} \rtimes \sigma_{a,c}^+)+
    L(\ds{-a}{c}\rtimes\sigma_{b})+
     \\
        &\sigma_{b,c,a}^- + 
                            L(\ds{-b}{c}\times
                                \ds{\frac{1}{2}}{a} 
                                    \rtimes \sigma)+
\\
&L(\ds{\frac{1}{2}}{a}\rtimes \sigma_{b,c}^+)
        +
        L(\ds{-a}{b}\rtimes \sigma_c)
        +
        \\
        &L(\ds{\frac{1}{2}}{b}\rtimes \sigma_{a,c}^-)
        +
        L(\ds{-b}{c}\rtimes \sigma_a)+
\\
&\sigma_{a,b,c}^+ +L(\ds{\frac{1}{2}}{a}\rtimes \sigma_{b,c}^-).
\end{align*}
\end{theorem}
\begin{proof}
By \eqref{long-intertwining}, Theorems \ref{muic-A3} and \ref{ciganovic-A2},
and Propositions 
\ref{prop-kompozicijski-fakt2}
and
\ref{prop-kompozicijski-fakt2}, we listed all irreducible subquotients, up to multiplicities. 
Proposition 
\ref{diskretne-u-velikoj-prpopzicija}
shows multiplicity one for discrete series.
Propositions \ref{prop-l-pola-a-b-c},
\ref{propozicija L ab, sigma c} and 
\ref{prop-mult-L-bc-a}, show multiplicity one for four remaining subquotients, which appear with a multilpicity of more than one, on the right hand side of \eqref{long-intertwining}.
\end{proof}



\section{Composition series of
$\ds{\frac{1}{2}}{b}
         \rtimes \sigma_{a,c}^+$ and $\ds{\frac{1}{2}}{b} \rtimes \sigma_{a,c}^-$.}
\label{kompozicijski1}

Here we determine the composition series of the kernel $K_1$, from 
Section \ref{dekompozicija1}. For the first representation, we show that a discrete series is a subrepesentation, and use intertwining operators to position other subquotients.

\begin{lemma} 
\label{jacquet-diskretne-u-malo-vecoj}
    We have in $R(G)$, with maximum multiplicities:
    \begin{align*}
    \mu^*(\ds{\frac{1}{2}}{b}&\times\ds{\frac{1}{2}}{a}\times \sigma_c)
    \geq
    2\cdot \ds{-a}{b}\otimes \sigma_c.
    \end{align*}
\end{lemma}
\begin{proof}
    Consider
    $0\leq s\leq r \leq b+\frac{1}{2}$, $0\leq j\leq i \leq a+\frac{1}{2}$,
    $\delta'\otimes \sigma'\leq \mu^*(\sigma_c)$, and
    \begin{align*}
    \ds{r-b}{-\frac{1}{2}}
    \times \ds{b-s+1}{b}
    \times
    \ds{i-a}{-\frac{1}{2}}
    \times
    \ds{a-j+1}{a}
     \times \delta'&
    \\
    \otimes 
    \ds{b+1-r}{b-s}
    \times \ds{a+1-i}{a-j}
    \rtimes \sigma'&.
    \end{align*}    
    By
    \eqref{jacquet-strogo-pozitivna} we have 
    $\delta'\otimes\sigma'=1\otimes\sigma_c$.
    So $i=j$, $r=s$ and we have
    \begin{align*}
     \ds{r-b}{-\frac{1}{2}}
        \times
        \ds{b-r+1}{b}
        \times \ds{a-i+1}{a}
        \times
        \ds{i-a}{-\frac{1}{2}}
        \otimes \sigma_c.
    \end{align*}
Two options are  
 $r-b=-a$, $i-a=\frac{1}{2}$ 
and 
 $i-a=-a$, $r-b=\frac{1}{2}$.
\end{proof}
The next proposition gives positions of both 
$\sigma_{b,c,a}^+$ and 
$ L(\ds{-a}{b}\rtimes \sigma_c )$.
\begin{proposition} 
\label{lema-ulaganje-tri-reprezentacije}
We have embeddings
\begin{align*}
\sigma_{b,c,a}^+ 
&\hookrightarrow
\ds{-a}{b}\rtimes \sigma_c / \sigma_{a,b,c}^- 
\\
&\hookrightarrow 
\ds{\frac{1}{2}}{b}\rtimes \sigma_{a,c}^+.
\end{align*}
\end{proposition}
\begin{proof}
By Theorem 2-6. of \cite{muic-hanzer:langzel} 
and Lemma \ref{lema-diskretna-podreprezentacija}
we have an epimorphism
\begin{equation*}
\ds{-a}{-\frac{1}{2}}\rtimes \sigma_c \twoheadrightarrow \sigma_{a,c}^+.
\end{equation*}
Now we have a composition of an embedding and an epimorphism
\begin{equation} \label{prvo-ulaganje-u-A+}
\begin{split}
\ds{-a}{b}\rtimes \sigma_c 
&\rightarrow
\ds{\frac{1}{2}}{b}\times \ds{-a}{-\frac{1}{2}}\rtimes \sigma_c
\\
&\rightarrow
\ds{\frac{1}{2}}{b}\rtimes \sigma_{a,c}^+.
\end{split}
\end{equation}
By Proposition \ref{druga}, and 
 \ref{diskretne-u-velikoj-prpopzicija} and
Lemma
\ref{jacquet-diskretne-u-malo-vecoj}, all representations in 
\eqref{prvo-ulaganje-u-A+}, 
have $\sigma_{b,c,a}^+$ as a subquotient, with multiplicity one, and the last doesn't contain
$\sigma_{a,b,c}^-$. 
\end{proof}

Now we want to see a position of 
$L(\ds{\frac{1}{2}}{a}
\rtimes \sigma_{b,c}^+)$.

\begin{lemma}
\label{lema-L12abc-jacq}
    We have in $R(G)$, with multiplicity one:
\begin{equation}
\label{lema-L12abc-jacq-f0}
\mu^*(
L(\ds{\frac{1}{2}}{a} \rtimes \sigma_{b,c}^+ )
)
    \geq 
    \ds{-b}{c}
        \otimes 
            L(\ds{\frac{1}{2}}{a}\rtimes \sigma).
\end{equation}
\end{lemma}
\begin{proof}
By Proposition 2.4 of  \cite{ciganovic}, 
and Proposition \ref{druga},  we have
    \begin{align}
    \label{lema-L12abc-jacq-f1}
    \ds{\frac{1}{2}}{a} \rtimes \sigma_{b,c}^+
    &=
        \sigma_{b,a,c}^+
        +
        L(\ds{\frac{1}{2}}{a} \rtimes \sigma_{b,c}^+ ),
    \\
    \label{lema-L12abc-jacq-f2}
    \ds{-a}{b}\rtimes  \sigma_c
    &=
        \sigma_{b,c,a}^+ +
        \sigma_{a,b,c}^- +
        L(\ds{-a}{b}\rtimes  \sigma_c).
    \end{align}
    It is enough to show that considered subquotient
    appears, in the appropriate Jacquet module, in 
    \eqref{lema-L12abc-jacq-f1}, with multiplicity one, 
    but does not appear in \eqref{lema-L12abc-jacq-f2}.
 
    $\bullet$
    In
    \eqref{lema-L12abc-jacq-f1}, we search for 
    $0\leq s\leq r\leq c+\frac{1}{2} $ and 
    $\delta'\otimes \sigma'\leq \mu^*(\sigma_{b,c}^+)$ such that
    \begin{align*}
        &\ds{-b}{c}
        \leq
        \ds{r-a}{-\frac{1}{2}}
            \times \ds{a+1-s}{a}
                \times \delta' \quad \textrm{and}
        \\
        &L(\ds{\frac{1}{2}}{a}\rtimes \sigma)
        \leq 
        \ds{a+1-r}{a-s}\rtimes \sigma'.         
    \end{align*}
    We see that $\nu^{-b}\rho$, 
    $\nu^{b}\rho$ and $\nu^{c}\rho$ are in the cuspidal support of $\delta'$.  By \eqref{jacquet-segment},    
    $\delta'\otimes \sigma'
        =\ds{-b}{c}\otimes \sigma$, once.
    Now
    $r=a+\frac{1}{2}$ and $s=0$.

    $\bullet$
    In \eqref{lema-L12abc-jacq-f2}, 
    we search for
    $0\leq s\leq r\leq b+a+1 $ and 
    $\delta'\otimes \sigma'\leq \mu^*(\sigma_c)$ such that
    \begin{align*}
        &\ds{-b}{c}
        \leq
        \ds{r-b}{a}
            \times \ds{b+1-s}{b}
                \times \delta' \quad \textrm{and}
        \\
        &L(\ds{\frac{1}{2}}{a}\rtimes \sigma)
        \leq 
        \ds{b+1-r}{b-s}\rtimes \sigma'.         
    \end{align*}
    By
    \eqref{jacquet-strogo-pozitivna}
    $\nu^{-b}\rho$ is not in s cuspidal support of $\delta'$.
    So $r=0$ and $s=0$. 
    By
    \eqref{jacquet-strogo-pozitivna},
    the second equation can not be satisfied.
\end{proof}

\begin{lemma}
\label{lema-diskretna-i-L-u-vecoj}
    We have in $R(G)$, with maximum multiplicities:
    \begin{align*}
    \mu^*(\ds{\frac{1}{2}}{b}\times\ds{\frac{1}{2}}{a}\times \sigma_c)
    \geq
    \ds{\frac{1}{2}}{a}\otimes \sigma_{b,c}^+&
    \\
    +
    \ds{-b}{c}
        \otimes 
            L(\ds{\frac{1}{2}}{a}\rtimes \sigma)&.
    \end{align*}
\end{lemma}
\begin{proof}
    By \eqref{komnozenje} consider
    $0\leq s\leq r \leq b+\frac{1}{2}$, $0\leq j\leq i \leq a+\frac{1}{2}$,
    $\delta'\otimes \sigma'\leq \mu^*(\sigma_c)$, and
    \begin{align*}
    \ds{r-b}{-\frac{1}{2}}
    \times \ds{b-s+1}{b}
    \times
    \ds{i-a}{-\frac{1}{2}}
    \times
    \ds{a-j+1}{a}
     \times \delta'
    \\
    \otimes 
    \ds{b+1-r}{b-s}
    \times \ds{a+1-i}{a-j}
    \rtimes \sigma'.
    \end{align*}    
For the first subquotient
by \eqref{jacquet-strogo-pozitivna}, we see $\delta'\otimes\sigma'=1\otimes\sigma_c$.
Further $r=b+\frac{1}{2}$, $s=0$, 
$i=a+\frac{1}{2}$ 
and  $j=a+\frac{1}{2}$.
For the second,
we must have $r=0$, so $s=0$ and $i=a+\frac{1}{2}$.
Now  $a-j+1=\frac{1}{2}$, is not possible, since $\sigma'$ is in discrete series, by 
\eqref{jacquet-strogo-pozitivna}. So 
$\nu^\frac{1}{2}\rho$ is in cuspidal support of $\delta'$ and we have 
$\delta'\otimes \sigma'=\ds{\frac{1}{2}}{a}\otimes\sigma$, $j=0$.
\end{proof}

The next proposition gives position of 
$L(\ds{\frac{1}{2}}{a}
\rtimes \sigma_{b,c}^+)$
\begin{proposition} 
\label{prop-L12abc}
    There exists an embedding
    \[
        \ds{\frac{1}{2}}{a}
         \rtimes \sigma_{b,c}^+
            \hookrightarrow
        \ds{\frac{1}{2}}{b}
         \rtimes \sigma_{a,c}^+.
    \]
\end{proposition}
\begin{proof}
    Denote 
    $\pi=\ds{\frac{1}{2}}{a} \times 
       \ds{\frac{1}{2}}{b}
        \rtimes \sigma_c$.
    By Lemma
    \ref{lema-diskretna-podreprezentacija}, we have
\begin{align*}
   \ds{\frac{1}{2}}{b}
         \rtimes \sigma_{a,c}^+
            \hookrightarrow
    \ds{\frac{1}{2}}{b}\times \ds{\frac{1}{2}}{a}\rtimes \sigma_c \cong \pi,& \quad 
    \textrm{ and by \eqref{lema-pola-a-bc-f} }
    \\
    \sigma_{b,c,a}^+
    +
    L
    (\ds{\frac{1}{2}}{a} \rtimes \sigma_{b,c}^+ )
    \overset{R(G)}{=}
    \ds{\frac{1}{2}}{a}
    \rtimes \sigma_{b,c}^+
        \hookrightarrow \pi.&
\end{align*}
It is enough to show that both 
$\sigma_{b,c,a}^+$
    and
$\textrm{L}
(\ds{\frac{1}{2}}{a}\rtimes \sigma_{b,c}^+ )$
appear in 
$\pi$ once. 
This follows by  
Lemmas 
\ref{lema-pola-a-bc-tvrdnja},
\ref{lema-L12abc-jacq}
and
\ref{lema-diskretna-i-L-u-vecoj}.
\end{proof}
Finally
\begin{proposition} \label{prop-pola-b-ac-+}
    Induced representation
    $\ds{\frac{1}{2}}{b}
         \rtimes \sigma_{a,c}^+$ has a unique irreducible subrepresentation and unique quotient. We have an exact sequence
    \begin{align*}
        L(\ds{\frac{1}{2}}{a}\rtimes \sigma_{b,c}^+)
        +
        L(\ds{-a}{b}\rtimes \sigma_c)
    &\longrightarrow
    \\
        \ds{\frac{1}{2}}{b}
            \rtimes \sigma_{a,c}^+
                /
                \sigma_{a,b,c}^+
     &\longrightarrow
         L(\ds{\frac{1}{2}}{b}
         \rtimes \sigma_{a,c}^+).
    \end{align*}    
\end{proposition}
\begin{proof}
    Composition factors are determined by 
    Proposition
    \ref{prop-kompozicijski-fakt1}.
    Position of irreducible subquotients are determined by
    Propositions \ref{druga}, 
    \ref{lema-ulaganje-tri-reprezentacije},
    \ref{lema-pola-a-bc-tvrdnja}
    and \ref{prop-L12abc}.
\end{proof}

By Theorem \ref{contra}, we have

\begin{corollary}
\label{kor_-b_-pola-ac-+}
Induced representation
    $\ds{-b}{-\frac{1}{2}}
         \rtimes \sigma_{a,c}^+$ has a unique irreducible subrepresentation and a unique irreducible quotient. We have an exact sequence
    \begin{align*}
        L(\ds{\frac{1}{2}}{a}\rtimes \sigma_{b,c}^+)
        +
        L(\ds{-a}{b}\rtimes \sigma_c)
    &\longrightarrow
    \\
        \ds{-b}{-\frac{1}{2}}
            \rtimes \sigma_{a,c}^+
                /
                L(\ds{\frac{1}{2}}{b}
         \rtimes \sigma_{a,c}^+)
     &\longrightarrow
         \sigma_{a,b,c}^+.
    \end{align*}    
\end{corollary}

Now we state the
composition series of $\ds{\frac{1}{2}}{b} \rtimes \sigma_{a,c}^-$ as a direct consequence of Proposition \ref{prop-kompozicijski-fakt2}.
\begin{proposition} \label{prop-pola-b-ac--}
 We have a  non split exact sequence
    \begin{align*}
L(\ds{\frac{1}{2}}{a}\rtimes \sigma_{b,c}^-)
\rightarrow
\ds{\frac{1}{2}}{b} \rtimes \sigma_{a,c}^-
\rightarrow
L(\ds{\frac{1}{2}}{b} \rtimes \sigma_{a,c}^-).     
    \end{align*}
\end{proposition}
By Theorem \ref{contra}, we have
\begin{corollary} \label{kor_-b_-pola_ac-}
 We have a  non split exact sequence
    \begin{align*}
L(\ds{\frac{1}{2}}{b} \rtimes \sigma_{a,c}^-)
\rightarrow
\ds{-b}{-\frac{1}{2}} \rtimes \sigma_{a,c}^-
\rightarrow
L(\ds{\frac{1}{2}}{a}\rtimes \sigma_{b,c}^-).     \end{align*}
\end{corollary}


\section{
Composition series of 
$\ds{-c}{b}\times \ds{\frac{1}{2}}{a}\rtimes \sigma$
}
\label{kompozicijski4}
Here we determine the composition series of the kernel $K_2$ from Section \ref{dekompozicija1}.

\begin{proposition}
\label{prop_-cb_pola_a}
    Induced representation
    $\ds{-c}{b}\times \ds{\frac{1}{2}}{a}\rtimes \sigma$
    has exactly one irreducible subrepresentation, and two irreducible quotients. We have an exact sequence
    \begin{align*}
        \sigma_{b,c,a}^+ 
        + 
        L(\ds{-b}{c}\times \ds{\frac{1}{2}}{a}\times \sigma)
        +
        \sigma_{b,c,a}^-
        \longrightarrow
        \\
        \ds{-c}{b}\times \ds{\frac{1}{2}}{a}\rtimes \sigma
        \textrm{/}
        L(\ds{-b}{c}\rtimes \sigma_a)
        \\
        \longrightarrow
        L(\ds{\frac{1}{2}}{a}\rtimes \sigma_{b,c}^+)
        +
        L(\ds{\frac{1}{2}}{a}\rtimes \sigma_{b,c}^-).
    \end{align*}
\end{proposition}

\begin{proof}
Denote the induced representation by $\pi$.
By Theorem \ref{ciganovic-A2}, $\pi$
is a multiplicity one representation
with irreducible subquotients listed in the claim.
Consider embeddings
\begin{align*}
    \ds{-c}{b}\rtimes \sigma_a \rightarrow \pi,
    \\
    \ds{\frac{1}{2}}{a}
    \times L(\ds{-b}{c}\rtimes \sigma) 
    \rightarrow \pi,
    \\
    \ds{-b}{c}\rtimes 
        L(\ds{\frac{1}{2}}{a} \rtimes \sigma)
        \rightarrow
        \pi^\wedge,
     \\
     \ds{-a}{-\frac{1}{2}} \rtimes \sigma_{b,c}^\pm
     \rightarrow
        \pi^\wedge.   
\end{align*}
We shall describe representations on the left and the claim will follow.
\\
$\bullet$
By Proposition \ref{druga}, 
$\ds{-c}{b}\rtimes \sigma_a$
has a unique irreducible subrepresentation: 
\\
$L(\ds{-b}{c}\rtimes \sigma_a)$, and two irreducible quotients: $\sigma_{b,c,a}^\pm$.
\\
$\bullet$ By Corrolary 4.1 of 
\cite{ciganovic}, 
$\ds{\frac{1}{2}}{a}
    \times L(\ds{-b}{c}\rtimes \sigma)$ is a quotient of 
    $\ds{\frac{1}{2}}{a}
    \times \ds{-b}{c}\rtimes \sigma$ containing
    a unique irreducible subrepresentation: 
    $L(\ds{-b}{c}\rtimes \sigma_a)$ and a 
    quotient:
    $L( \ds{-b}{c}
    \times\ds{\frac{1}{2}}{a}
    \rtimes \sigma )
    $.
\\
$\bullet$ By Corrolary 4.1 of 
\cite{ciganovic}, 
$\ds{-b}{c}\rtimes 
        L(\ds{\frac{1}{2}}{a} \rtimes \sigma)$
is a quotient of 
$\ds{\frac{1}{2}}{a}
    \times \ds{-b}{c}\rtimes \sigma$, containing 
    a unique irreducible quotient:
    $L(\ds{\frac{1}{2}}{a}
    \times 
    \ds{-b}{c}\rtimes \sigma)$ and two irreducible subrepresentations:
    $L(\ds{\frac{1}{2}}{a}\rtimes \sigma_{b,c}^\pm)$
    Thus $\pi$ has a quotient, containing two irreducible quotients: 
    $L(\ds{\frac{1}{2}}{a}\rtimes \sigma_{b,c}^\pm)$
    and a unique irreducible subrepresentation:
    $L(\ds{\frac{1}{2}}{a}
    \times 
    \ds{-b}{c}\rtimes \sigma)$.
\\
$\bullet$ Similarly, by Lemma 
\ref{lema-pola-a-bc-tvrdnja},
$\pi$ has a quotient, containing a unique irreducible subrepresentation
$\sigma_{b,c,a}^\pm$
and
a unique irreducible quotient  
$L(\ds{\frac{1}{2}}{a} \rtimes \sigma_{b,c}^\pm )$.   
\end{proof}
By Theorem \ref{contra}, we have
\begin{corollary}
\label{kor_-bc_pola_a}
    Induced representation
    $\ds{-b}{c}\times \ds{-a}{-\frac{1}{2}}\rtimes \sigma$
    has exactly
    two irreducible subrepresentations and 
    one irreducible quotient. We have an exact sequence
    \begin{align*}
        &\sigma_{b,c,a}^+ 
        + 
        L(\ds{-b}{c}\times \ds{\frac{1}{2}}{a}\times \sigma)
        +
        \sigma_{b,c,a}^-
        \longrightarrow
        \\
        &\ds{-b}{c}\times \ds{-a}{-\frac{1}{2}}\rtimes \sigma
        \\
        &\textrm{\ }\qquad \qquad \textrm{/}
        (
                L(\ds{\frac{1}{2}}{a}\rtimes \sigma_{b,c}^+)
            +
                L(\ds{\frac{1}{2}}{a}\rtimes \sigma_{b,c}^-)
        )
        \\
        &\longrightarrow
        L(\ds{-b}{c}\rtimes \sigma_a).
    \end{align*}
\end{corollary}


\section{Composition series of 
$\ds{-a}{c}\rtimes\sigma_{b}$}
\label{kompozicijski3}

Here we determine composition series of 
the kernel $K_3\cong \ds{-c}{a}\rtimes\sigma_{b}$ from Section \ref{dekompozicija1}.
First we determine 
subrepresentations.

\begin{proposition} 
\label{lema-jedinstvena-diskretna-podrep}
We have a unique subrepresentation
\[
\sigma_{b,c,a}^+ 
\hookrightarrow 
\ds{-a}{c}\rtimes \sigma_b.
\]
\end{proposition}

\begin{proof}
Compare
Theorem \ref{muic-A3}
and Proposition 
\ref{lema-ulaganje-tri-reprezentacije}
with 
\[
\ds{-a}{c}\rtimes \sigma_b
\hookrightarrow
\ds{\frac{1}{2}}{b} \times
\ds{-a}{c}\rtimes \sigma.
\] 
\end{proof}

Now we determine position of
$L(\ds{-b}{c}\rtimes \sigma_a)$.

\begin{lemma}
\label{multipliciteti-u-pola-c-ab+}
    We have in $R(G)$, with maximum multiplicities
    \begin{align*}
    \ds{\frac{1}{2}}{c}
        \times\ds{\frac{1}{2}}{b} \rtimes \sigma_a
    &\geq \sigma_{b,c,a}^+ +\sigma_{b,c,a}^- 
        +L(\ds{-b}{c}\rtimes \sigma_a),
    \\
    \ds{\frac{1}{2}}{c}\rtimes \sigma_{a,b}^+
    &\geq
    \sigma_{b,c,a}^+ +0\cdot \sigma_{b,c,a}^- 
        +L(\ds{-b}{c}\rtimes \sigma_a). 
    \end{align*}
\end{lemma}

\begin{proof}
    By Proposition \ref{druga}, 
    $\sigma_{b,c,a}^\pm$ appears in the first formula. For the multiplicity,
    by \eqref{komnozenje}, consider
    $0\leq s\leq r \leq c+\frac{1}{2}$, $0\leq j\leq i \leq b+\frac{1}{2}$,
    $\delta'\otimes \sigma'\leq \mu^*(\sigma_a)$, and
    \begin{equation*}
    \begin{split}
    \ds{r-c}{-\frac{1}{2}}
    \times \ds{c-s+1}{c}
    \times
    \ds{i-b}{-\frac{1}{2}}
    \\
    \times
    \ds{b-j+1}{b}
     \times \delta'   
    \otimes 
    \ds{c+1-r}{c-s}
    \times \ds{b+1-i}{b-j}
    \rtimes \sigma'.
    \end{split}
    \end{equation*}   
By \eqref{lema-pola-a-bc-f1} we search for
$\ds{\frac{1}{2}}{a}\otimes \sigma_{b,c}^\pm$.
Now  $r=c+\frac{1}{2}$, $s=0$, $i=b+\frac{1}{2}$, $j=0$, so
$\delta'\otimes \sigma'=\ds{\frac{1}{2}}{a}\otimes \sigma$, and apply \eqref{lemma-diskr-u-dva-seg-kusp-f1}.
For $L(\ds{-b}{c}\rtimes \sigma_a)$
use Lemma \ref{lema-jacq-Lbc-a} to search for 
$\ds{\frac{1}{2}}{a}\otimes L(\ds{-b}{c}\rtimes \sigma)$ in the same way.
For the second equation consider
$0\leq s\leq r \leq c+\frac{1}{2}$, 
    $\delta'\otimes \sigma'\leq \mu^*(\sigma_{a,b}^+)$ and
\begin{equation*}
    \begin{split}
    \ds{r-c}{-\frac{1}{2}}
    \times \ds{c-s+1}{c}
     \times \delta'   
    \otimes 
    \ds{c+1-r}{c-s}
    \rtimes \sigma'.
    \end{split}
    \end{equation*}  
Look for  $\ds{\frac{1}{2}}{a}\otimes \sigma_{b,c}^\pm$.
So  $r=c+\frac{1}{2}$, $s=0$, and
$\delta'=\ds{\frac{1}{2}}{a}$. In \eqref{jacquet-segment} we have 
$i=-\frac{1}{2}$ and $j=b$, so $\sigma'= \sigma_b $. But
$\sigma_{b,c}^-\nleq \ds{\frac{1}{2}}{a}\rtimes \sigma_b$, by
Lemma \ref{lema-diskretna-podreprezentacija}. 
For $L(\ds{-b}{c}\rtimes \sigma_a)$, use already proved and
$\mu^*(\sigma_{a,b}^+)\geq \ds{\frac{1}{2}}{a}\otimes \sigma_b$.
\end{proof}

\begin{lemma}
\label{ulaganje-ac-b-u-pola-c-ab+}
We have an embedding
\[
\ds{-a}{c}\rtimes \sigma_b
\hookrightarrow
\ds{\frac{1}{2}}{c}\rtimes \sigma_{a,b}^+.
\]    
\end{lemma}

\begin{proof}
    By Lemma \ref{lema-diskretna-podreprezentacija}
    we have composition of an embedding and an epimorphism
    \begin{align*}
    \ds{-a}{c}\rtimes \sigma_b
    &\hookrightarrow
    \ds{\frac{1}{2}}{c}
            \times
            \ds{-a}{-\frac{1}{2}}\rtimes \sigma_b
    \\
    &\twoheadrightarrow
    \ds{\frac{1}{2}}{c}\rtimes \sigma_{a,b}^+.
    \end{align*}
Lemmas 
\ref{lema-diskretna-podreprezentacija} and
\ref{lema-pola-a-bc-tvrdnja}, 
Theorem 
\ref{muic-A3},
and Lemma 
\ref{multipliciteti-u-pola-c-ab+}, 
imply that all representations here have 
$\sigma_{b,c,a}^+$ as a subquotient, with multiplicity one.
By Proposition \ref{lema-jedinstvena-diskretna-podrep}, 
$\sigma_{b,c,a}^+$
is a unique irreducible subrepresentation of 
$\ds{-a}{c}\rtimes \sigma_b$. The claim follows.
\end{proof}

\begin{lemma}
\label{ulaganje-bc-a/bca-upola-c-ab}
    We have an embedding
    \[
        \ds{-b}{c}\rtimes \sigma_a /\sigma_{b,c,a}^- 
    \hookrightarrow 
        \ds{\frac{1}{2}}{c} \rtimes \sigma_{a,b}^+.
    \]
\end{lemma}
\begin{proof}
Consider a composition of an embedding and an epimorphism
\begin{align*}
        \ds{-b}{c}\rtimes \sigma_a  
    &\hookrightarrow 
               \ds{\frac{1}{2}}{c}\times\ds{-b}{-\frac{1}{2}}
        \rtimes
                \sigma_a
    \\            
    &\twoheadrightarrow
        \ds{\frac{1}{2}}{c} \rtimes \sigma_{a,b}^+.
\end{align*}
By Proposition \ref{druga}
$\ds{-b}{c}\rtimes  \sigma_a=
\sigma_{b,c,a}^+ +
\sigma_{b,c,a}^- +
L(\ds{-b}{c}\rtimes \sigma_a),
$
with discrete series being subrepresentations.
Apply Lemma \ref{multipliciteti-u-pola-c-ab+}.
\end{proof}

The next proposition give a position of 
$L(\ds{-b}{c}\rtimes \sigma_a)$.
\begin{proposition} \label{prop-ulaganje-d-ac_b_prvo}
    We have an embedding
    \[
        \ds{-b}{c}\rtimes \sigma_a /\sigma_{b,c,a}^- 
    \hookrightarrow 
        \ds{-a}{c} \rtimes \sigma_b.
    \]
\end{proposition}
\begin{proof}
    By Proposition \ref{druga} and Theorem \ref{muic-A3}
    we have in $R(G)$
    \[
    \ds{-b}{c}\rtimes \sigma_a \textrm{/} \sigma_{b,c,a}^-
    \leq
    \ds{-a}{c}\rtimes \sigma_b.
    \]
By Lemmas \ref{ulaganje-ac-b-u-pola-c-ab+}
and
\ref{ulaganje-bc-a/bca-upola-c-ab}, these representations embedd into 
$\ds{\frac{1}{2}}{c} \rtimes \sigma_{a,b}^+$, which has irreducible subquotiens of the first representations with multiplicity one, by Lemma \ref{multipliciteti-u-pola-c-ab+}.
\end{proof}

Finally we determine  position of 
$L(\ds{-a}{b}\rtimes \sigma_c)$.

\begin{proposition}
\label{prop-ulaganje-d-ac_b_drugo}
We have an embedding 
\[
    \ds{-a}{b}\rtimes \sigma_c
    \textrm{/}
    \sigma_{a,b,c}^-
\hookrightarrow
    \ds{-a}{c}\rtimes \sigma_b.
\]
   
\end{proposition}
\begin{proof}
    By Proposition 
    \ref{lema-ulaganje-tri-reprezentacije} we have
  \[
  \ds{-a}{b}\rtimes \sigma_c
  \textrm{/}
  \sigma_{a,b,c}^-
  \hookrightarrow
  \ds{\frac{1}{2}}{b}\rtimes \sigma_{a,c}^+,
  \textrm{ so}
  \]
  \begin{align*}
    \ds{-a}{b}\rtimes \sigma_c \textrm{/}
    \sigma_{a,b,c}^-
        \hookrightarrow
            &\ds{-a}{c}\times\ds{\frac{1}{2}}{b}\rtimes \sigma, 
                \textrm{ and }
    \\
     \ds{-a}{c}\rtimes\sigma_b           
        \hookrightarrow
            &\ds{-a}{c}\times\ds{\frac{1}{2}}{b}\rtimes \sigma.
  \end{align*}
The claim follows, 
since by Theorem
\ref{glavni-rezultat} the representation on the right is multiplicity one, 
and by
Proposition \ref{druga} and Theorem \ref{muic-A3} we have in $R(G)$
\[
\ds{-a}{b}\rtimes \sigma_c \textrm{/}
    \sigma_{b,a,c}^-
\leq 
\ds{-a}{c}\rtimes\sigma_b. 
\]
\end{proof}

Now we write composition series for 
$\ds{-a}{c}\rtimes\sigma_{b}$.

\begin{proposition}
\label{prop_-ac_b}
    Induced representation $\ds{-a}{c}\rtimes\sigma_{b}$ has a unique irreducible subrepresentation. We have an exact sequence
    \begin{align*}
        L(\ds{-b}{c}\rtimes \sigma_a)
        +
        L(\ds{-a}{b}\rtimes \sigma_c)
        &\longrightarrow
        \\
        \ds{-a}{c}\rtimes\sigma_{b}/\sigma_{a,b,c}^+
        &\longrightarrow
        L(\ds{-a}{c}\rtimes\sigma_{b}).
    \end{align*}
\end{proposition}
\begin{proof}
    Composition factors are determined by 
    Theorem
    \ref{muic-A3}.
    Positions of irreducible subquotients are determined by
    Propositions
    \ref{druga}, 
    \ref{prop-ulaganje-d-ac_b_prvo}
    and 
    \ref{prop-ulaganje-d-ac_b_drugo}.
\end{proof}

By Theorem \ref{contra}, we have
\begin{corollary}
\label{kor_-ca_b}
    Induced representation $\ds{-c}{a}\rtimes\sigma_{b}$ has a unique irreducible quotient. We have an exact sequence
    \begin{align*}
        L(\ds{-b}{c}\rtimes \sigma_a)
        +
        L(\ds{-a}{b}\rtimes \sigma_c)
        &\longrightarrow
        \\
        \ds{-c}{a}\rtimes\sigma_{b}/
        L(\ds{-a}{c}\rtimes\sigma_{b})
        &\longrightarrow
        \sigma_{a,b,c}^+
        .
    \end{align*}
\end{corollary}



\section{The main result}
\label{dekompozicija2}
Here we give the main result, a composition series of the  representation $\psi$.
\begin{theorem} 
Let $\psi=\delta([\nu^{-a}\rho,\nu^c\rho])\times \delta([\nu^\frac{1}{2}\rho,\nu^b\rho])\rtimes \sigma$ and define representations
\begin{align*}
    W_1=&\sigma_{b,c,a}^+ +L(\ds{\frac{1}{2}}{a}\rtimes \sigma_{b,c}^-),
        \\
    W_2=&L(\ds{\frac{1}{2}}{a}\rtimes \sigma_{b,c}^+)
        +
        L(\ds{-a}{b}\rtimes \sigma_c)
        +
        \\
        &L(\ds{\frac{1}{2}}{b}\rtimes \sigma_{a,c}^-)
        +
        L(\ds{-b}{c}\rtimes \sigma_a),
        \\
     W_3=&L(\ds{\frac{1}{2}}{b} \rtimes \sigma_{a,c}^+)+
            L(\ds{-a}{c}\rtimes\sigma_{b})+
        \\
          &\sigma_{b,c,a}^- + 
                            L(\ds{-b}{c}\times
                                \ds{\frac{1}{2}}{a} 
                                    \rtimes \sigma),
        \\
      W_4=&L(\psi).                            
\end{align*}
Then there exists a sequence $\{0\}=V_0\subseteq V_1
\subseteq V_2 \subseteq V_3
\subseteq V_4=\psi$,
such that
\begin{equation*}
V_i/V_{i-1}\cong W_i,\quad  i=1,\ldots,4.
\end{equation*}    
Further,  $W_1$ is chosen to be the largest possible, then $W_2$, and so on.
\end{theorem}
\begin{proof}
We use the notation $K_i$, $H_i$, and $f_i$ from Section \ref{dekompozicija1}.
Composition series of $K_1, K_2$ and $K_3$ are determined by Propositions \ref{prop-pola-b-ac-+}, \ref{prop-pola-b-ac--},
\ref{prop_-cb_pola_a}, and Corollary \ref{kor_-ca_b}.
Composition series of $H_1, H_2$ and $H_3$ are determined by Proposition \ref{prop_-ac_b} and 
Corollaries \ref{kor_-bc_pola_a}, 
\ref{kor_-b_-pola_ac-} and \ref{kor_-b_-pola-ac-+}.
For all $i\geq 1$ denote
\begin{align*}
k_i&=K_i  
\cap { Im (f_{i-1}\circ \cdots \circ f_0)  },   
 \quad
h_i=H_i  
\cap{ Im (g_{i-1}\circ \cdots \circ g_0)  }.   
\end{align*}
After subtracting, taking only summands of classes of irreducible representations, with non-negative signs, we have in $R(G)$
\begin{alignat*}{3}
    k_1&=K_1,
    \quad 
    k_2=&
    \lfloor
     K_2-k_1 \rfloor_{R^+_0(G)},
    \quad
    k_3=& \lfloor K_3-k_2 - k_1 \rfloor_{R^+_0(G)},
\\
    h_1&=H_1,
    \quad 
    h_2=&
    \lfloor H_2-h_1 \rfloor_{R^+_0(G)},
    \quad
    h_3=& \lfloor H_3-h_2 - h_1 \rfloor_{R^+_0(G)}.
\end{alignat*}
We have $k_1\cong K_1$ and $h_1\cong H_1$.
Further, calculating composition factors of $k_2$ and $h_2$ and
comparing with composition series of $K_2$ and $H_2$, we see that 
$k_2$ and $h_2$ have exactly two irreducible quotients. Similarly, we determine $K_3$ and $H_3$. So we have
exact sequences, and no irreducible subquotient can go on lower position:
\begin{align}
\label{fka}    
 \begin{split}
 &L(\ds{\frac{1}{2}}{a}\rtimes \sigma_{b,c}^+)
        +
        L(\ds{-a}{b}\rtimes \sigma_c)
        +
        L(\ds{\frac{1}{2}}{b}\rtimes \sigma_{a,c}^-)
        \rightarrow
\\
        &k_1
            /
            (\sigma_{a,b,c}^+ 
                +L(\ds{\frac{1}{2}}{a}\rtimes \sigma_{b,c}^-)
                    \rightarrow
                        L(\ds{\frac{1}{2}}{b}
                            \rtimes \sigma_{a,c}^+).
\end{split}
\\   
\label{fkb}
&k_2/L(\ds{-b}{c}\rtimes \sigma_a)
             \cong 
                \sigma^-_{b,c,a} 
                +
                L(\ds{-b}{c}\times 
                    \ds{\frac{1}{2}}{a}
                        \rtimes \sigma),
\\
\label{fkc}
&k_3=L(\ds{-a}{c}\rtimes \sigma_b),
\\
\label{fha}
    \begin{split}
    &L(\ds{-b}{c}\rtimes \sigma_a)
        +
        L(\ds{-a}{b}\rtimes \sigma_c)
        \rightarrow
        \\
        &h_1/\sigma_{a,b,c}^+
        \rightarrow
        L(\ds{-a}{c}\rtimes\sigma_{b}),
    \end{split}
\\
\label{fhb}  
\begin{split}
&h_2/( 
    L(\ds{\frac{1}{2}}{a} \rtimes \sigma_{b,c}^+)
    +
    L(\ds{\frac{1}{2}}{a} \rtimes \sigma_{b,c}^-)
    )
    \cong
    \\
    &
    \qquad \qquad 
    \sigma_{b,c,a}^-
            + 
    L(\ds{-b}{c}\times
        \ds{\frac{1}{2}}{a} 
            \rtimes \sigma),
\end{split}
\\
\label{fhc}   
&h_3\cong 
    L(\ds{\frac{1}{2}}{b} \rtimes \sigma_{a,c}^+)
    +
    L(\ds{\frac{1}{2}}{b} \rtimes \sigma_{a,c}^-).
\end{align}
We define representations $V_i$, $i=1,\ldots,4$ as follows. 
By \eqref{fka} \( V_1:=W_1
 \hookrightarrow \psi\).
Further, \eqref{fka} and \eqref{fha} show      
\(
 W_2 \hookrightarrow \psi/V_1.
\)
Let $V_2$ be the preimage of $W_2$, in $\psi$. Denote representations
\begin{align*}
    \zeta=&L(\ds{\frac{1}{2}}{b} \rtimes \sigma_{a,c}^+), \\
    \nu=&L(\ds{-a}{c}\rtimes\sigma_{b}), \\
    M=&\sigma_{b,c,a}^- + 
                            L(\ds{-b}{c}\times
                                \ds{\frac{1}{2}}{a} 
                                    \rtimes \sigma), \\
     \tau=&\psi/V_2.
\end{align*}
By \eqref{fka} and \eqref{fha} we have embeddings
\begin{align} \label{dva-ulaganja-u-tau}
    \zeta \hookrightarrow \tau \hookleftarrow \nu.
\end{align}
By \eqref{fka} and \eqref{fkb} we have an embedding and an epimorphism
\[
M \hookrightarrow \tau/\zeta \overset{P_\zeta}{\leftarrow} \tau,
\quad \textrm{ thus }\quad  P_\zeta^{-1}(M)/\zeta \cong M. 
\]
By \eqref{fha} and \eqref{fhb} we have an embedding and an epimorphism
\[
M \hookrightarrow \tau/\nu \overset{P_\nu}{\leftarrow} \tau,
\quad \textrm{ thus }\quad  P_\nu^{-1}(M)/\nu \cong M.
\]
By Proposition \ref{prop-kompozicijski-fakt2}, 
$\psi$ is a multiplicity one, and so is $\tau$ and an embedding
\[
    P_\zeta^{-1}(M)
        /
        (P_\zeta^{-1}(M) \cap P_\nu^{-1}(M))
        \longrightarrow \tau /P_\nu^{-1}(M)
\]
shows that in $R(G)$:
\begin{align*}
  &M+\zeta-P_\zeta^{-1}(M) \cap P_\nu^{-1}(M)\leq \tau-M-\nu,\quad \textrm{ so } 
  \\
  &M+\zeta+\nu \leq P_\zeta^{-1}(M) \cap P_\nu^{-1}(M)
                    +(\tau-M).
\end{align*}
We conclude 
  $  M\cong P_\zeta^{-1}(M) \cap P_\nu^{-1}(M)$,
and have an embedding
\begin{align} \label{ulaganje-M}
    M
    \hookrightarrow \tau.
\end{align}
Combining \eqref{dva-ulaganja-u-tau} and \eqref{ulaganje-M} we have
an embedding
\begin{align*}
W_3 \hookrightarrow \psi/ V_2.
\end{align*}
Let $V_3$ be the preimage of $W_3$ in $\psi$.
We see that in $R(G)$:
\[
\psi\overset{R(G)}{=} W_1+W_2+W_3 +L(\psi).
\]
We proved the filtration formula. Now we show the last claim, about maximality.
Decompositions of $k_1$ and $h_1$, 
\eqref{fka} and \eqref{fha}, show that no irreducible subquotient of $W_2$, can be a subrepresentation of $\psi$. They also show that
$L(\ds{\frac{1}{2}}{b}\rtimes \sigma_{a,c}^+)$ and $L(\ds{-a}{c}\rtimes \sigma_b)$ can not be embedded into $\psi/V_1$. To see the same for factors of $M$, first assume that $\sigma_{b,c,a}^-\hookrightarrow \psi/V_1$. Since
$k_1 /V_1 \hookrightarrow \psi/V_1$, and $k_1$ doesn't contain $\sigma_{b,c,a}^-$,
we obtain 
$\sigma_{b,c,a}^-\hookrightarrow \psi/k_1$. On the other hand
$k_2 \hookrightarrow \psi/k_1$, and $k_2$ contains 
$\sigma_{b,c,a}^-$, but not as a subrepresentation. Since $\psi$ is a multiplicity one, we got a contradiction. Similarly for the other factor of $M$.
\end{proof}

\newpage

\def\cprime{$'$} \def\cprime{$'$}
\providecommand{\bysame}{\leavevmode\hbox to3em{\hrulefill}\thinspace}
\providecommand{\MR}{\relax\ifhmode\unskip\space\fi MR }
\providecommand{\MRhref}[2]{
  \href{http://www.ams.org/mathscinet-getitem?mr=#1}{#2}
}
\providecommand{\href}[2]{#2}

\end{document}